\documentclass[reqno]{amsart}

\makeatletter
\@namedef{subjclassname@2020}{%
  \textup{2020} Mathematics Subject Classification}
\makeatother

\usepackage[english]{babel}
\usepackage[T1]{fontenc}
\usepackage[utf8]{inputenx}

\usepackage{a4wide}
\usepackage{color}
\usepackage{mathrsfs}
\usepackage{mathtools}
\usepackage{amsmath}
\usepackage{amssymb}
\usepackage{bbm}
\usepackage{esint}
\numberwithin{equation}{section}
\usepackage[colorlinks,citecolor=green,linkcolor=red]{hyperref}
\usepackage{todonotes}

\newcommand{\N}{\mathbb{N}}
\newcommand{\R}{\mathbb{R}}
\newcommand{\C}{{\rm C}}
\newcommand{\AC}{{\rm AC}}
\renewcommand{\S}{{\rm S}}
\newcommand{\D}{{\rm D}}
\newcommand{\W}{{\rm W}}
\newcommand{\sfd}{{\sf d}}
\renewcommand{\d}{{\mathrm d}}
\newcommand{\X}{{\rm X}}
\newcommand{\Y}{{\rm Y}}
\newcommand{\mm}{\mathfrak{m}}
\newcommand{\1}{\mathbbm 1}
\newcommand{\Mod}{{\rm Mod}}
\newcommand{\UG}{{\rm UG}}
\newcommand{\LIP}{{\rm LIP}}

\newcommand{\loc}{{\rm loc}}

\newcommand{\fr}{\penalty-20\null\hfill\(\blacksquare\)}

\newtheorem{theorem}{Theorem}[section]
\newtheorem{corollary}[theorem]{Corollary}
\newtheorem{lemma}[theorem]{Lemma}
\newtheorem{proposition}[theorem]{Proposition}
\newtheorem{definition}[theorem]{Definition}

\newtheorem{remark}[theorem]{Remark}

\title{Pullback of a quasiconformal map between arbitrary metric measure spaces}

\author[Toni Ikonen]{Toni Ikonen}
\address[Toni Ikonen]{University of Jyvaskyla, Department of Mathematics and Statistics, P.O.\ Box 35 (MaD), FI-40014
University of Jyvaskyla, Finland}
\email{toni.m.h.ikonen@jyu.fi}

\author[Danka Lu\v{c}i\'{c}]{Danka Lu\v{c}i\'{c}}
\address[Danka Lu\v{c}i\'{c}]{Universit\`{a} di Pisa,
Dipartimento di Matematica, Largo Bruno Pontecorvo 5,
56127 Pisa, Italy}
\email{danka.lucic@dm.unipi.it}

\author[Enrico Pasqualetto]{Enrico Pasqualetto}
\address[Enrico Pasqualetto]{Scuola Normale Superiore, Piazza dei
Cavalieri, 7, 56126 Pisa, Italy.}
\email{enrico.pasqualetto@sns.it}

\date{\today}

\begin{document}
\begin{abstract}
We prove that every (geometrically) quasiconformal homeomorphism between
metric measure spaces induces an isomorphism between
the cotangent modules constructed by Gigli. We obtain this by first
showing that every continuous mapping \(\varphi\) with bounded outer
dilatation induces a pullback map \(\varphi^*\) between
the cotangent modules of Gigli, and then proving the functorial nature of the
resulting pullback operator. Such pullback is
consistent with the differential for metric-valued locally Sobolev
maps introduced by Gigli--Pasqualetto--Soultanis.

Using the consistency between Gigli's and Cheeger's cotangent modules for PI spaces, we prove that quasiconformal homeomorphisms between PI spaces preserve the dimension of Cheeger charts, thereby generalizing earlier work by Heinonen--Koskela--Shanmugalingam--Tyson.

Finally, we show that if $\varphi$ is a given homeomorphism with bounded outer dilatation, then $\varphi^{-1}$ has bounded outer dilatation if and only if $\varphi^{*}$ is invertible and $\varphi^{-1}$ is Sobolev. In contrast to the setting of Euclidean spaces, Carnot groups, or more generally, Ahlfors regular PI spaces, the Sobolev regularity of $\varphi^{-1}$ needs to be assumed separately.
\end{abstract}
\maketitle
\section{Introduction}
The prime objective of this paper is to fit the metric theory of
quasiconformal maps initiated by Heinonen--Koskela \cite{Hei:Kos:98}
into the framework of the differential calculus on metric measure
spaces introduced by Gigli \cite{Gig:18}. Before stating our main
results, let us remind some terminology.
\medskip

Let \((\X,\sfd_\X,\mm_\X)\), \((\Y,\sfd_\Y,\mm_\Y)\) be
metric measure spaces and fix $p \in (1,\infty)$. We say that a continuous map $\varphi\colon\X\to\Y$
has \emph{bounded outer dilatation} if there exists a constant $K > 0$ such that $\Mod_{p}( \Gamma ) \leq K\Mod_{p}( \varphi\Gamma )$ for all path 
families $\Gamma$; here $\Mod_{p}$ is the conformal $p$-modulus, see \eqref{eq:Mod}. When $\varphi$ is a homeomorphism, it was observed by Williams \cite[Theorem 1.1]{Wil:12} that the aforementioned modulus inequality is equivalent to requiring that $\varphi$ is locally Sobolev, i.e., $\varphi \in \D^{1,p}_\loc(\X,\Y)$, and satisfies the key inequality
\begin{equation}\label{eq:essentialinequality}
	 \int_{ \varphi^{-1}(E) } \rho_\varphi^p \,\d\mm_{\X} \leq K \mm_{\Y}( E ) \quad\text{for every Borel $E \subset \Y$},
\end{equation}
where \(\rho_\varphi\) stands for the \emph{minimal weak upper gradient} of $\varphi$. Notice that the above definitions depend on the fixed exponent $p$. By modifying Williams' proof, it was recently observed in \cite{Nta:Rom:21} that the equivalence between the modulus inequality and \eqref{eq:essentialinequality} also holds in the case $p = 2$ for all continuous mappings between metric surfaces, see \cite{Nta:Rom:21} for the relevant definitions. We observe here that the argument fully generalizes to the metric measure space setting for all $1 < p < \infty$, see Proposition \ref{lemm:energydistortion}.

There are several classes of important non-injective mappings of bounded outer dilatation. For example, one might consider $z \mapsto z^{2}$ in the complex plane (or other $N$-to-$1$ quasiregular mappings, see the monograph \cite{Ric:93}), folding maps $
\mathbb{R}^{n} \times \mathbb{R} \colon (x,y) \mapsto (x,|y|)$ (or other mappings of bounded deformation, see \cite{Gig:18,Gig:Pas:Sou:20}), or monotone mappings in 
some metric applications; see \cite{Lyt:Wen:17,Iko:Rom:ar,Mei:Wen:21,Nta:Rom:21,Iko:21} for some recent work. Many of our results apply to all continuous mappings of bounded outer 
dilatation. We also mention that an orientation-preserving homeomorphism $f \colon ( \R^n, \| \cdot \|, \mathcal{L}^n ) \to ( \R^n, \| \cdot \|, \mathcal{L}^n )$ is said to be a mapping with \emph{locally integrable distortion (dilatation)} if there exists a locally integrable weight $\omega \colon \R^n \to \left[1,\infty\right]$ such that the inverse $\varphi = f^{-1}$ has bounded outer dilatation as a mapping from $( \R^n, \| \cdot \|, \mathcal{L}^{n} )$ to $( \R^n, \| \cdot \|, \omega \mathcal{L}^n )$, for $p = n$. We refer the interested reader to \cite{Ast:Iwa:Mar:09,Ma:Vl:Sr:Ya:09}, and references therein, for further reading about mappings with locally integrable distortion.
\medskip

Another important topic in Metric Geometry is the differential
calculus recently developed by Gigli in \cite{Gig:18}. One of
the goals of this work was to show that a Sobolev function
\(f\in\W^{1,p}(\X)\), albeit a priori characterised only in
terms of weak upper gradients (which act as the absolute value
of the differential), actually comes with a linear notion
of differential in a rather natural way. The relevant concept
is that of \emph{cotangent module}, whose construction makes
sense in the full generality of arbitrary metric measure spaces,
even when an underlying linear structure is missing.
The cotangent module \(L^p(T^*\X)\) associated with a given metric
measure space \((\X,\sfd_\X,\mm_\X)\) can be described as the
completion of the \(L^\infty\)-linear combinations of the
`formal' differentials of Sobolev functions \(f \colon \X \to \R\), which are required to
obey the expected calculus rules. Technically speaking, the
structure one has to use in order to define \(L^p(T^*\X)\) is
the notion of \emph{\(L^p\)-normed \(L^\infty\)-module} over \(\X\).
An \(L^p\)-normed \(L^\infty\)-module is a Banach space that can be
`localised on Borel subsets of \(\X\)', in a suitable sense.
The \emph{differential} \(\d\colon\W^{1,p}(\X)\to L^p(T^*\X)\)
is then a linear map satisfying \(|\d f|=\rho_f\) for every \(f\).
\medskip

The core of this paper is Subsection \ref{ss:pullback} (more
specifically, Theorem \ref{thm:pullback_map}), where we prove
that every continuous map \(\varphi\colon\X\to\Y\) with bounded outer
dilatation induces a pullback operator at the
level of cotangent modules. Namely, there is a unique linear
and continuous map \(\varphi^*\colon L^p(T^*\Y)\to L^p(T^*\X)\)
such that \(\varphi^*(h\,\d f)=h\circ\varphi\,\d(f\circ\varphi)\)
whenever \(f\colon\Y\to\R\) is Sobolev and \(h\colon\X\to\R\)
is bounded Borel. Some care needs to be taken in formulating such an identity since the resulting operator must depend
only on the equivalence classes of \(f\) and \(h\) up to
\(\mm_\Y\)-a.e.\ equality. Heuristically, the reason why \(\varphi^*\)
is well-defined is due to the nice composition properties that are
enjoyed by \(\varphi\): as we shall see later, it holds that if \(f\colon\Y\to\R\) is
Sobolev, then \(f\circ\varphi\colon\X\to\R\) is Sobolev as well
and \(\rho_{f\circ\varphi}\leq\rho_f\circ\varphi\,\rho_\varphi\) $\mm_{\X}$-almost everywhere.
A first important consequence of this phenomenon is the following:
given two Sobolev functions \(f,g\colon\Y\to\R\) which agree
\(\mm_\Y\)-almost everywhere, we have that \(\d(f\circ\varphi)\)
and \(\d(g\circ\varphi)\) coincide; the differentials 
$\d( f \circ \varphi )$ and $\d( g \circ \varphi )$ degenerate on a Borel subset 
containing $\left\{ f \circ \varphi \neq g \circ \varphi \right\}$. In view of this 
observation, the definition \(\varphi^*(\d f)\coloneqq\d(f\circ\varphi)\)
is well-posed. Furthermore, the key bound \(\rho_{f\circ\varphi}\leq
\rho_f\circ\varphi\,\rho_\varphi\) entails the validity of the
inequality \(|\varphi^*\omega|\leq\rho_\varphi\,|\omega|\circ\varphi\)
for all \(\omega\in L^p(T^*\Y)\) and thus the continuity of
\(\varphi^*\). 

The motivation behind our construction of the pullback operator induced
by a quasiconformal map comes from \cite[Theorem 10.8]{HKST:01},
where the authors built a notion of pullback operator for
quasiconformal maps between Ahlfors regular spaces supporting
a weak Poincar\'{e} inequality; in their approach, differentials of
Sobolev functions are intended in the sense of Cheeger \cite{Ch:99}.
By virtue of the consistency between Cheeger's differential and
Gigli's differential (cf.\ \cite[Section 2.5]{Gig:18}), our main
Theorem \ref{thm:pullback_map} can be de facto regarded as a
generalisation of \cite[Theorem 10.8]{HKST:01}.

The aforementioned consistency between Cheeger's and Gigli's differential 
and Theorem \ref{thm:qc_vs_dimension} shows the following: quasiconformal homeomorphisms 
between PI spaces preserve dimensions of Cheeger charts. This means that if $(U, L)$ 
and $( \varphi(U), L' )$ are Cheeger charts of a Borel set $U \subset \X$ and $\varphi(U)$, 
respectively, then the images of $L$ and $L'$ have the same dimension. In fact, similar 
results hold for the Sobolev charts recently introduced in \cite{EB:S:21}. In Section 
\ref{sec:consistency}, we show that for any given continuous mapping of bounded outer 
dilatation, the pullback $\varphi^{*}$ is consistent with a related pullback construction 
from \cite[Definition 3.4]{Gig:Pas:Sou:20}.

\medskip

In Subsections \ref{ss:nec_QC} and \ref{ss:suf_QC} we employ
the pullback operator to provide necessary and sufficient conditions
for quasiconformality. About the former, it is easy to show that if the map \(\varphi\) is a quasiconformal homeomorphism, 
i.e., if $\varphi$ and its inverse have bounded outer dilatation,
then the associated pullback operator \(\varphi^*\) is an isomorphism
of Banach spaces (see Corollary \ref{cor:pullback_QC}). Consequently, in Theorem
\ref{thm:qc_vs_refl} we point out that if two metric measure
spaces \(\X\) and \(\Y\) are quasiconformally equivalent, then
\(L^p(T^*\X)\) is reflexive if and only if \(L^p(T^*\Y)\) is
reflexive. We recall that if $L^{p}( T^{*}\X )$ is reflexive, 
also the Sobolev space ${\rm W}^{1,p}(\X)$ is reflexive.
Concerning the sufficient conditions, we prove in Theorem
\ref{thm:equiv_QC} that a map \(\varphi\) with bounded outer
dilatation is automatically quasiconformal as soon as the
pullback operator \(\varphi^*\) is invertible and
\(f\circ\varphi^{-1}\) is Sobolev for every \(f\colon\X\to\R\)
boundedly-supported Lipschitz. This second condition cannot be disposed of, as we show in Remark \ref{rem:failure}.
\medskip

\textbf{Acknowledgements.}
The authors would like to thank Tapio Rajala for having suggested the problem, as well as Kai Rajala and Pekka Koskela for helpful
comments on a preliminary version of the paper.

The first named author was supported by the Academy of Finland, project number 308659, and 
by the Vilho, Yrjö and Kalle Väisälä Foundation. The second and third named authors were 
supported by the Academy of Finland (project number 314789).
The second named author was also supported by the project
2017TEXA3H ``Gradient flows, Optimal Transport and Metric
Measure Structures", funded by the Italian Ministry of
Research and University. The third named author was also
supported by the Balzan project led by Luigi Ambrosio.
\section{Preliminaries}\label{sec:prelim}
In this section we gather all the terminology and results
we will need in the main Section \ref{sec:main}. The material
contained in Subsections \ref{ss:curves_Mod}, \ref{ss:Sobolev},
and \ref{ss:QC} is mostly taken from \cite{HKST:15} (see also
\cite{Bj:Bj:11}), while Subsection \ref{ss:cotg_mod} is taken
from \cite{Gig:18}.
\subsection{Curves and modulus}\label{ss:curves_Mod}
In this paper, by a \emph{metric measure space} we mean a triple
\((\X,\sfd,\mm)\), where \((\X,\sfd)\) is a complete and separable
metric space, while \(\mm\geq 0\) is a boundedly finite Borel measure
on \((\X,\sfd)\).
We denote by \(\mathscr B(\X)\) the Borel \(\sigma\)-algebra of \(\X\).
By a \emph{curve} in \(\X\) we mean a continuous map
\(\gamma\colon I_\gamma\to\X\), where \(I_\gamma\subset\R\) is a
compact interval. We denote by \(\C(\X)\) the family of all curves
in \(\X\). We say that a curve \(\gamma\) is \emph{rectifiable} if there exists $L > 0$ such that for each increasing sequence $( t_i )_{ i = 1 }^{ n }$, with $t_0$ and $t_n$ being the end points of $I_{\gamma}$, we have
\begin{equation*}
	\sum_{ i = 1 }^{ n } \sfd( \gamma( t_{i-1} ), \gamma( t_i ) )
	\leq
	L.
\end{equation*}
The smallest constant $L$ with this property is the \emph{length} $\ell( \gamma )$ of $\gamma$. For a rectifiable curve, the \emph{metric speed} \(|\gamma'|\) is the limit
\[
|\gamma'|(t)=\lim_{s\to t}\frac{\sfd(\gamma(t),\gamma(s))}{|t-s|},
\quad\text{ for a.e.\ }t\in I_\gamma.
\]
The limit is well-defined almost everywhere; see \cite{Dud:07}.

We say that a curve \(\gamma\) in \(\X\) is
\emph{absolutely continuous} provided there exists a function
\(g\in L^1(I_\gamma)\) such that \(\sfd(\gamma(t),\gamma(s))
\leq\int_s^t g(r)\,\d r\) holds for every \(s,t\in I_\gamma\)
with \(s<t\). Such a curve is rectifiable and the smalllest function $g$ coincides with \( |\gamma'| \). We denote by \(\AC(\X)\) the family of all absolutely
continuous curves in \(\X\).

We recall that for every rectifiable curve $\gamma$ in $\X$, there exists an absolutely continuous curve $\widetilde{\gamma} \colon \left[0, \ell(\gamma)\right] \to \X$ and a surjective non-decreasing function $\psi \colon I_\gamma \to \left[0, \ell( \gamma )\right]$ for which $\gamma = \widetilde{\gamma} \circ \psi$, where $| \widetilde{\gamma}' |( s ) = 1$ for almost every $s$ \cite{HKST:15}. We refer to $\widetilde{\gamma}$ as the \emph{unit speed parametrization} of $\gamma$. For each Borel $\rho \colon \X \to \left[0,\infty\right]$, the \emph{path integral} of $\rho$ over $\gamma$ is defined by
\begin{equation*}
	\int_{ \gamma } \rho \,\d s \coloneqq \int_{ 0 }^{ \ell(\gamma) } \rho( \widetilde{\gamma}(t) ) \,\d t.
\end{equation*}
When $\gamma$ is absolutely continuous, the equality
\begin{equation*}
	\int_{ \gamma } \rho \,\d s
	=
	\int_{ I_{ \gamma } }
		\rho( \gamma(t) ) | \gamma' |(t)
	\,\d t
\end{equation*}
holds; see, for example, \cite{Dud:07} for a proof.
\smallskip

Later on, we will need another characterisation of absolutely
continuous curves. Before passing to its statement, let us introduce
a useful class of auxiliary functions: given a metric space
\((\X,\sfd)\), a point \(\bar x\in\X\), and a radius \(r>0\),
we define the function \(\eta_{\bar x,r}\colon\X\to[0,2r]\) as
\begin{equation}\label{eq:def_good_cut-off}
\eta_{\bar x,r}(x)\coloneqq\left\{\begin{array}{lll}
\sfd(x,\bar x),\\
2r-\sfd(x,\bar x),\\
0,
\end{array}\quad\begin{array}{lll}
\text{ if }x\in B_r(\bar x),\\
\text{ if }x\in B_{2r}(\bar x)\setminus B_r(\bar x),\\
\text{ if }x\in\X\setminus B_{2r}(\bar x).
\end{array}\right.
\end{equation}
The proof of the following criterion is essentially
contained in the proof of \cite[Theorem 1.1.2]{Amb:Gig:Sav:08}.
\begin{lemma}\label{lem:equiv_ms}
Let \((\X,\sfd)\) be a complete, separable metric space.
Let \(\gamma\colon[a,b]\to\X\) be a curve. Fix a dense
sequence \((x_n)_n\subset\X\) and denote \(\eta_{n,k}\coloneqq
\eta_{x_n,k}\) for all \(n,k\in\N\), where \(\eta_{x_n,k}\)
is given by \eqref{eq:def_good_cut-off}. Suppose
\(\eta_{n,k}\circ\gamma\colon[a,b]\to\R\) is absolutely continuous
for every \(n,k\in\N\) and \(g\in L^1(a,b)\), where
\[
g(t)\coloneqq\sup_{n,k\in\N}(\eta_{n,k}\circ\gamma)'(t),
\quad\text{ for a.e.\ }t\in[a,b].
\]
Then \(\gamma\) is absolutely continuous and the identity
\(|\gamma'|(t)=g(t)\) holds for a.e.\ \(t\in[a,b]\).
\end{lemma}
\begin{proof}
First of all, we claim that for any \(s,t\in[a,b]\) with \(s<t\)
it holds that
\begin{equation}\label{eq:equiv_ms}
\sfd(\gamma(t),\gamma(s))=\sup_{n,k\in\N}
\big(\eta_{n,k}(\gamma(t))-\eta_{n,k}(\gamma(s))\big).
\end{equation}
The inequality \(\geq\) follows from the fact that each function
\(\eta_{n,k}\) is \(1\)-Lipschitz. To prove the converse inequality,
fix any \(\varepsilon>0\). Choose \(n,k\in\N\) such that
\(\sfd(\gamma(s),x_n)<\varepsilon\) and \(\sfd(\gamma(t),\gamma(s))
+\varepsilon<k\). Then
\[
\eta_{n,k}(\gamma(t))-\eta_{n,k}(\gamma(s))=
\sfd(\gamma(t),x_n)-\sfd(\gamma(s),x_n)\geq
\sfd(\gamma(t),\gamma(s))-2\,\sfd(\gamma(s),x_n)\geq
\sfd(\gamma(t),\gamma(s))-2\varepsilon,
\]
whence it follows (thanks to the arbitrariness of \(\varepsilon>0\))
that the inequality \(\leq\) in \eqref{eq:equiv_ms} is verified.
The absolute continuity of \(\eta_{n,k}\circ\gamma\) ensures
that \(\eta_{n,k}(\gamma(t))-\eta_{n,k}(\gamma(s))
=\int_s^t(\eta_{n,k}\circ\gamma)'(r)\,\d r\), thus
\[
\int_s^t(\eta_{n,k}\circ\gamma)'(r)\,\d r\leq
\sfd(\gamma(t),\gamma(s))\overset{\eqref{eq:equiv_ms}}\leq
\int_s^t g(r)\,\d r,\quad\text{ for every }n,k\in\N\text{ and }
s,t\in[a,b],\,s<t.
\]
Given that \(g\in L^1(a,b)\) by assumption, we deduce that
\(\gamma\) is absolutely continuous. Moreover, an application
of the Lebesgue differentiation theorem yields
\begin{align*}
(\eta_{n,k}\circ\gamma)'(t)&=\lim_{s\nearrow t}\fint_s^t
(\eta_{n,k}\circ\gamma)'(r)\,\d r\leq\lim_{s\nearrow t}\frac{\sfd(\gamma(t),\gamma(s))}{t-s}=|\gamma'|(t)
\\
&\leq\lim_{s\nearrow t}\fint_s^t g(r)\,\d r=g(t),\quad
\text{ for a.e.\ }t\in[a,b].
\end{align*}
By taking the supremum over \(n,k\in\N\), we finally conclude
that \(g(t)=|\gamma'|(t)\) for a.e.\ \(t\in[a,b]\).
\end{proof}
Given a curve family \( \Gamma \subset \C( \X ) \), a function $\rho \colon \X \rightarrow \left[0,\infty\right]$ is \emph{admissible} for $\Gamma$ if $\rho$ is Borel and $\int_{ \gamma } \rho \,\d s \geq 1$ for every rectifiable $\gamma \in \Gamma$. Given any exponent \(p\in(1,\infty)\) and a curve family
\(\Gamma\subset\C(\X)\), the \emph{\(p\)-modulus} of \(\Gamma\)
is defined as
\begin{equation}\label{eq:Mod}
{\rm Mod}_p(\Gamma)\coloneqq\inf\bigg\{\int\rho^p\,\d\mm\;\bigg|\;
\rho\colon\X\to[0,+\infty]\text{ admissible for $\Gamma$}
\bigg\}.
\end{equation}
A given property \(\mathcal P=\mathcal P(\gamma)\) is said to hold
for \emph{\({\rm Mod}_p\)-almost every \(\gamma\in\C(\X)\)} provided
there exists a curve family \(\Gamma\subset\C(\X)\) with
\({\rm Mod}_p(\Gamma)=0\) such that \(\mathcal P(\gamma)\)
is verified for every \(\gamma\in\C(\X)\setminus\Gamma\).
\medskip

Given any \(p\in[1,\infty]\), we denote by \(\mathcal L^p(\mm)\) the
set of all \(p\)-integrable real-valued functions defined on \(\X\),
while \(L^p(\mm)\) stands for the space of all equivalence classes of
elements of \(\mathcal L^p(\mm)\) up to \(\mm\)-a.e.\ equality.
Given a Borel function \(f\colon\X\to\R\), we denote by \(\pi_\X(f)\)
its equivalence class up to \(\mm\)-a.e.\ equality. Accordingly, it
holds  that \(L^p(\mm)=\pi_\X\big(\mathcal L^p(\mm)\big)\). For
brevity, let us denote
\[
\1_E^\mm\coloneqq\pi_\X(\1_E),\quad
\text{ for every }E\in\mathscr B(\X).
\]
Moreover,
we define \(\mathcal L^p_\loc(\mm)\) as the set of all Borel functions \(f\colon\X\to\R\)
that are locally \(p\)-integrable, meaning that every point
\(x\in\X\) has an open neighbourhood \(U_x\) such that
\(\1_{U_x}f\in\mathcal L^p(\mm)\). We also set
\(L^p_\loc(\mm)\coloneqq\pi_\X\big(\mathcal L^p_\loc(\mm)\big)\).
The space \({\sf Sf}(\X)\subseteq L^\infty(\mm)\) of
\emph{simple functions} on \(\X\) is defined as
\[
{\sf Sf}(\X)\coloneqq\bigg\{\sum_{i=1}^n\lambda_i\1_{E_i}^\mm
\;\bigg|\;n\in\N,\;(\lambda_i)_{i=1}^n\subseteq\R,\;(E_i)_{i=1}^n
\text{ Borel partition of }\X\bigg\}.
\]
It can be readily checked that \({\sf Sf}(\X)\) is dense in \(L^\infty(\mm)\)
with respect to the strong topology, namely, the topology induced by the \(L^\infty(\mm)\)-norm.

\begin{remark}\label{rmk:finite_on_curves}{\rm
Let \((\X,\sfd,\mm)\) be a metric measure space. Let
\(g\in\mathcal L^p_{\rm loc}(\mm)\), \(g\geq 0\) be given. Then
\[
\int_{ \gamma  } g\,\d s<+\infty,\quad
\text{ for }{\rm Mod}_p\text{-a.e.\ }\gamma\in{\rm C}(\X).
\]
To prove this, fix a countable open covering \((U_n)_n\) of \(\X\)
such that \(\int_{U_n}g^p\,\d\mm<+\infty\) for every \(n\in\N\).
Define \(E_n\coloneqq U_n\setminus\bigcup_{1\leq i<n}U_i\)
for every \(n\in\N\) and
\[
\tilde g\coloneqq\sum_{n\in\N}\alpha_n\1_{E_n}g,
\quad\text{ where we set }\alpha_n\coloneqq
\bigg(2^n\int_{E_n}g^p\,\d\mm\bigg)^{-1/p}.
\]
Observe that \(\tilde g\in\mathcal L^p(\mm)\). Consider the
curve family \(\Gamma\coloneqq\big\{\gamma\in \C(\X)\,:\,
\int_{ \gamma } g\,\d s = +\infty\big\}\).
Now fix \(\gamma\in\Gamma\). Since \(\gamma(I_\gamma)\) is
compact, there exists a finite set \(F\subset\N\) such that
\(\gamma(I_\gamma)\subset\bigcup_{n\in F}E_n\) and thus
\(\int_{ \gamma  }\tilde g \,\d s
\geq\big(\min_{n\in F}a_n\big)\int_{ \gamma  }g\,\d s=+\infty\). Therefore, we have that
\(\varepsilon\tilde g\) is a competitor for \({\rm Mod}_p(\Gamma)\)
for every \(\varepsilon>0\), so that accordingly
\({\rm Mod}_p(\Gamma)\leq\varepsilon^p\int\tilde g^p\,\d\mm\)
for every \(\varepsilon>0\). By letting \(\varepsilon\searrow 0\)
we finally conclude that \({\rm Mod}_p(\Gamma)=0\), proving the claim.
\fr}\end{remark}
\subsection{Weak upper gradients}\label{ss:Sobolev}
Let us begin by recalling the fundamental concept
of weak upper gradient.
\begin{definition}[Weak upper gradient]\label{defi:weakupper}
Let \((\X,\sfd_\X,\mm)\) be a metric measure space and \((\Y,\sfd_\Y)\)
a metric space. Fix \(p\in(1,\infty)\) and a map
\(h\colon\X\to\Y\). A Borel function
\(g\colon\X\to[0,+\infty]\) is a
\emph{weak upper gradient} of \(h\) if for
\(\Mod_p\)-a.e.\ \(\gamma\in\C(\X)\) it holds that
\begin{equation}\label{eq:uppergradient:inequality}
\sfd_\Y\big(h(\gamma(b)),h(\gamma(a))\big)
\leq\int_\gamma g\,\d s,\quad
\text{ where }I_\gamma=[a,b].
\end{equation}
We denote by \(\UG^p_\loc(h)\) and \(\UG^p(h)\)
the sets of all weak upper gradients of \(h\) that in \(\mathcal L^p_\loc(\mm)\) and \(\mathcal L^p(\mm)\),
respectively.
\end{definition}
Each element of $\UG^{p}_{\loc}( h )$ is a pointwise decreasing limit of elements of $\UG^{p}_{\loc}( h )$ for which \eqref{eq:uppergradient:inequality} holds for every path $\gamma$. More precisely, we have the following.
\begin{lemma}[Upper gradients]\label{lemma:strongupper}
Let $p \in (1,\infty)$, $h \colon \X \to \Y$, and $g \in \UG^{p}_{\loc}( h )$. Then there exists a nonnegative $G_{0} \in \mathcal{L}^{p}( \mm_{\X} )$ such that $g_{\varepsilon} = g + \varepsilon G_{0}$ satisfies \eqref{eq:uppergradient:inequality} for every rectifiable $\gamma \in \C(\X)$ for every $\varepsilon > 0$.
\end{lemma}
\begin{proof}
Let $g \in \UG^{p}_{\loc}( h )$ and $\mathcal{N} \subset \C( \X )$ such that \eqref{eq:uppergradient:inequality} holds for every rectifiable $\gamma \in \C( \X ) \setminus \mathcal{N}$, with $\Mod_{p}( \mathcal{N} ) = 0$. Observe that $\mathcal{N}$ does not contain any constant curves. Hence \cite[Lemma 5.2.8]{HKST:15} shows that there exists a nonnegative $G_{0} \in \mathcal{L}^{p}( \mm_{\X} )$ such that $\int_{ \gamma } G_0 \,\d s = \infty$ for every {rectifiable $\gamma \in \mathcal{N}$. Setting $g_{\varepsilon} = g + \varepsilon G_{0}$, for every $\varepsilon > 0$, establishes the claim.}
\end{proof}
The following fact is well-known, but we prove it for the sake of completeness.
\begin{proposition}\label{lemm:localityupper}
Let $p \in (1,\infty)$, $h \colon \X \to \Y$, and $g \in \mathcal{L}^{p}_{\loc}( \mm_{\X} )$ nonnegative. Then $g \in \UG_{\loc}( h )$ if and only if for 
${\rm Mod}_p$-almost every $\gamma \in \AC( \X )$, we have
\begin{itemize}
	\item $\int_{ \gamma } g \,\d s < \infty$,
	\item $h \circ \gamma$ is absolutely continuous, and
	\item $| ( h \circ \gamma )' |(t) \leq g( \gamma(t) ) | \gamma' |(t)$ for almost every $t \in I_{\gamma}$.
\end{itemize}
\end{proposition}
\begin{proof}
The "if"-direction follows by integrating the inequality $| ( h \circ \gamma )' |(t) \leq g( \gamma(t) ) | \gamma' |(t)$, using the absolute continuity of $h \circ \gamma$.

In the "only if"-direction, we first consider the family $\mathcal{N}$ of paths $\gamma \in \AC( \X )$ which have one of the following properties
\begin{enumerate}
	\item $\int_{ I_{\gamma} } ( h \circ \gamma )(t) | \gamma' |(t) \,\d t = \infty$;
	\item there exists a closed interval $I \subset I_{\gamma}$ such that \eqref{eq:uppergradient:inequality} does not hold for the triple $( h, g, \gamma|_{I} )$.
\end{enumerate}
The paths satisfying the first property have negligible modulus due to Remark \ref{rmk:finite_on_curves}. The paths satisfying the second property also have negligible modulus since $g \in \UG_{\loc}^{p}( h )$. Hence $\Mod_{p}( \mathcal{N} ) = 0$. For every $\gamma \in \AC( \X ) \setminus \mathcal{N}$ and every $s, s' \in I_{\gamma}$, with $s \leq s'$, we have
\begin{equation*}
	\sfd_{\X}( h( \gamma(s) ), h( \gamma(s') ) )
	\leq
	\int_{s}^{s'}
		g( \gamma(t) ) | \gamma' |(t)
	\,\d t
	\leq
	\int_{ I_{\gamma} }
		g( \gamma(t) ) | \gamma' |(t) \,\d t < \infty.
\end{equation*}
The absolute continuity of the integral implies that $h \circ \gamma$ is continuous, has finite length, and is absolutely continuous. Since $h \circ \gamma$ is absolutely continuous, a modification of the above argument shows
\begin{equation*}
	\ell( h \circ \gamma|_{ \left[s,s'\right] } )
	=
	\int_{ s }^{ s' }
		|( h \circ \gamma )'|(t)
	\,\d t
	\leq
	\int_{s}^{s'}
			g( \gamma(t) ) | \gamma' |(t)
		\,\d t
	\quad\text{for every $s \leq s'$.}
\end{equation*}
The Lebesgue differentiation theorem establishes the inequality $| ( h \circ \gamma )' |(t) \leq g( \gamma(t) ) | \gamma' |(t)$ for almost every $t \in I_{\gamma}$.
\end{proof}
We now state some simple consequences of Proposition \ref{lemm:localityupper} (and of Lemma \ref{lemma:strongupper}).

\begin{lemma}[Gluing]\label{lemma:gluing}
Let $U, V \subset \X$ be open, $p \in (1,\infty)$, and $h \colon \X \to \Y$.
Given $g_{U}\in \UG^{p}_{\loc}( h|_{U} )$ and $g_{V} \in \UG^{p}_{\loc}( h|_{V} )$, 
denote $f(x) =\min\left\{ g_{U}(x), g_{V}(x) \right\}$ for  $x\in U \cap V$, 
$f(x) = g_{U}(x)$ for $x\in U \setminus V$, and $f(x) = g_{V}(x)$ for $x\in V \setminus U$. 
Then $f \in \UG^{p}_{\loc}( h|_{ U \cup V } )$.
\end{lemma}
\begin{corollary}\label{cor:upper}
Let $p \in (1,\infty)$, $h \colon \X \to \Y$ and $g \in \mathcal{L}^{p}_{\loc}( \mm_{\X} )$. If $g \in \UG^{p}_{\loc}( h )$, then $g|_{U} \in \UG^{p}_{\loc}( h|_{U} )$ for every open $U \subset \X$. 

Conversely, suppose the existence of an open cover $\left( U_{i} \right)_{ i = 1 }^{ \infty }$ of $\X$ and $g_{i} \in \UG^{p}_{\loc}( h|_{ U_{i} } )$ for each $i \in \mathbb{N}$. If $\widehat{g}_{i} \colon \X \rightarrow \left[0,\infty\right]$ is the zero extension of $g_{i}$ and $g = \sup_{i} \widehat{g}_{i}$ is the pointwise supremum, then $g \in \mathcal{L}^{p}_{\loc}( \X )$ implies $g \in \UG^{p}_{\loc}( h )$.
\end{corollary}

A map $h \colon \X \to \Y$ between a metric measure space $(\X, \sfd_{\X}, \mm_{\X} )$ and a metric space $( \Y, \sfd_{\Y} )$ is said to be \emph{measurable} provided there exists a set $N \subset \X$ with $\mm_{\X}( N ) = 0$ for which $h( \Y \setminus N )$ is separable and $h^{-1}( U )$ is $\mm_{\X}$-measurable for every open $U \subset \X$.
\begin{lemma}[Locality]\label{lemm:locality:dirichlet}
Let $p \in (1,\infty)$ and $h \colon \X \to \Y$ measurable. Then $\UG^{p}_{\loc}( h ) \neq \emptyset$ if and only if every $x \in \X$ is contained in an open set $U \subset \X$ such that $\UG^{p}( h|_{U} ) \neq \emptyset$.

Moreover, if $\UG^{p}_{\loc}( h ) \neq \emptyset$, then there exists a $g \in \UG^{p}_{\loc}( h )$ satisfying $g \leq g'$ $\mm_{\X}$-almost everywhere for every $g' \in \UG^{p}_{\loc}( h )$.
\end{lemma}
\begin{proof}
First, if $g \in \UG^{p}_{\loc}( h )$, the local finiteness of the measure $g^{p}\mm_{\X}$ implies that every $x \in \X$ is contained in some open set $U$ for which $\int_{U} g^{p} \,\d\mm_{\X} < \infty$. Corollary \ref{cor:upper} yields that
$g|_{U} \in \UG^{p}( h|_{U} )$ for such $U$.

In the converse direction, the Lindelöf property of $\X$ yields the existence of open sets $( U_{n} )_{ n = 1 }^{\infty} \subset \X$ covering $\X$ for which $\UG^{p}( h|_{ U_{n} } ) \neq \emptyset$ for each $n \in \mathbb{N}$. Let $V_{k} = \bigcup_{ n = 1 }^{ k } U_{n}$. It follows from Lemma \ref{lemma:gluing} that $\UG^{p}( h|_{ V_{k} } ) \neq \emptyset$ for every $k \in \mathbb{N}$. Theorem 6.3.20 \cite{HKST:15} shows the existence of $g_{k} \in \UG^{p}( h|_{ V_{k} } )$ such that $g_{k} \leq g'$ $\mm_{\X}$-almost everywhere for every $g' \in \UG^{p}( h|_{ V_{k} } )$. In fact, Lemma \ref{lemma:gluing} guarantees that $g_{k} \leq g'$ for every $g' \in \UG^{p}_{\loc}( h|_{ V_{k} } )$ for every $k \in \mathbb{N}$. The application of \cite[Theorem 6.3.20]{HKST:15} can be avoided by proving the closedness of $\UG^{p}( h|_{ V_{k} } )$ with respect to $L^{p}$-convergence. Having verified this and the convexity of $\UG^{p}( h|_{ V_{k} } )$, we then find an element $G \in \UG^{p}( h|_{ V_{k} } )$ with minimal $L^{p}$-norm. Lemma \ref{lemma:gluing} implies that for every other $g' \in \UG^{p}( h|_{ V_{k} } )$, we have $G \leq g'$ $\mm_{\X}$-almost everywhere. We may set $g_{k} \equiv G$.

Next, for each $k \in \mathbb{N}$, we consider the zero extension $\widehat{g}_{k}$ of $g_{k}$ and denote $g \coloneqq \sup_{ k } \widehat{g}_{k}$. Lemma \ref{lemma:gluing} implies that for every $k < l$, $g_{l}|_{ V_{k} } = g_{k}$ $\mm_{\X}$-almost everywhere, thus $g \in \UG^{p}_{\loc}(h)$. In particular, $\UG^{p}_{\loc}(h) \neq\emptyset$.

For the last part of the claim, we consider an arbitrary $g' \in \UG^{p}_{\loc}(h)$ and define $g$ as above. Corollary \ref{cor:upper} guarantees $g|_{V_{k}} = g_{k} \leq g'|_{ V_{k} }$ $\mm_{\X}$-almost everywhere for every $k \in \mathbb{N}$. Hence $g \leq g'$ $\mm_{\X}$-almost everywhere for every $g' \in \UG^{p}_{\loc}(h)$. The claim follows.
\end{proof}

\begin{definition}\label{rem:minimal:locality}
Let $h \colon \X \rightarrow \Y$ be measurable with $\UG^{p}_{\loc}( h ) \neq \emptyset$. If $g \in \UG^{p}_{\loc}( h )$ satisfies $g \leq g'$ pointwise $\mm_{\X}$-almost everywhere for every $g' \in \UG_{\loc}^{p}( h )$, we denote $\rho_{h} = g$ and call $\rho_{h}$ a \emph{minimal weak upper gradient} of $h$.
\end{definition}
\begin{lemma}\label{lemm:locality:minimal}
Let $g \in \mathcal{L}^{p}_{\loc}( \mm_{\X} )$ and $h \colon \X \rightarrow \Y$ measurable, with $\UG^{p}_{\loc}( h ) \neq \emptyset$. Then $g = \rho_{h}$ $\mm_{\X}$-almost everywhere if and only if $g|_{U} = \rho_{ h|_{U} }$ $\mm_{\X}$-almost everywhere for every open set $U \subset \X$.
\end{lemma}
Lemma \ref{lemm:locality:minimal} follows from Corollary \ref{cor:upper}.

\begin{remark}\label{rem:summation}{\rm
Let $\mathbb{V}$ be a Banach space and $u_{i} \colon \X \rightarrow \mathbb{V}$ a measurable mapping and $g_{i} \in \UG^{p}_{\loc}( u_{i} )$ for $i = 1,2$. Then $G = 2(g_{1}+g_{2}) \in \UG^{p}_{\loc}( u_{1} + u_{2} ) \cap \UG^{p}_{\loc}( u_{1} - u_{2} )$. This readily follows from Proposition \ref{lemm:localityupper}.
\fr}\end{remark}
\begin{lemma}\label{lem:differential}
Let $\mathbb{V}$ be a Banach space and $u_{i} \colon \X \rightarrow \mathbb{V}$ a measurable mappings with $\UG^{p}_{\loc}( u_i ) \neq \emptyset$ for $i = 1,2$. If $u_{1} = u_{2}$ on a Borel set $E \subset \X$, then $\rho_{ u_{1} - u_{2} } = 0$ $\mm_{\X}$-almost everywhere in $E$.
\end{lemma}
\begin{proof}
In  \cite[Proposition 6.3.22]{HKST:15} the following is stated: If $\UG^{p}( h_{1} ) \neq \emptyset \neq \UG^{p}( h_{2} )$ for some measurable mappings $h_{1}, h_{2} \colon \X \rightarrow \mathbb{V}$ and $h_{1} = h_{2}$ in a Borel set $E \subset \X$, then $\rho_{ h_{1} } = \rho_{ h_{2} }$ $\mm_{\X}$-almost everywhere in $E$. We apply this result as follows. Remark \ref{rem:summation} guarantees that $\UG^{p}_{\loc}( u_{1} - u_{2} ) \neq \emptyset$. Let $U \subset \X$ be open such that $\UG^{p}( ( u_1 - u_2 )|_{U} ) \neq \emptyset$, the existence of which is guaranteed by Lemma \ref{lemm:locality:dirichlet}. It follows from Lemma \ref{lemm:locality:minimal} that $\rho_{ u_1 - u_2 }|_{ U } = \rho_{ ( u_1 - u_2 )|_{U} }$ $\mm_{\X}$-almost everywhere in $U$. Now $( u_{1} - u_{2} )|_{U}$ is equal to the zero mapping $\mm_{\X}$-almost everywhere in $E \cap U$.  Then \cite[Proposition 6.3.22]{HKST:15} implies $\rho_{ ( u_{1} - u_{2} )|_{U} } = 0$ $\mm_{\X}$-almost everywhere in $E \cap U$. Since $\X$ can be covered by such sets $U$, the claim follows.
\end{proof}
Next, let $\Omega \subset \Y$ be open and $h \colon \Omega\rightarrow \rm{Z}$ a map from the metric measure space \( (\Y,\sfd_\Y,\mm_\Y) \) to a metric space 
$(\rm{Z},\sf d_{\rm Z})$. Then, for each $\varepsilon > 0$, let $\Gamma_{\varepsilon}( h )$ denote the collection of all $\gamma \in \rm{C}( \Omega )$ such that $h \circ \gamma \in \rm{C}( \rm{Z} )$ satisfies ${\sf d}_{\rm{Z}}( h( \gamma(a) ), h( \gamma(b) ) ) \geq \varepsilon$ for $I_{\gamma} = \left[a,b\right]$. Notice that even if $h$ is not continuous, we require the composition $h \circ \gamma$ to be continuous for each $\gamma \in \Gamma_{\varepsilon}(h)$.
\begin{proposition}\label{lem:dirichlet}
Let $\Omega \subset \Y$ be open, $p \in (1,\infty)$ and $h \colon \Omega \rightarrow \rm{Z}$ measurable. Then $\UG^{p}( h ) \neq \emptyset$ if and only if
\begin{equation}\label{eq:inequalitygood:measure:proof}
	\liminf_{ \varepsilon \rightarrow 0^{+} }
	\varepsilon^{p}
	\Mod_{p}\left( \Gamma_{\varepsilon}(h) \right)
	<
	\infty.
\end{equation}
Moreover, if the $\liminf$ is finite, then
\begin{equation}\label{eq:actuallimit}
	\int_{\Omega} \rho_{h}^{p} \,\d\mm_{\Y}
	=
	\liminf_{ \varepsilon \rightarrow 0^{+} }
	\varepsilon^{p}
	\Mod_{p}\left( \Gamma_{\varepsilon}(h) \right)
	=
	\lim_{ \varepsilon \rightarrow 0^{+} }
	\varepsilon^{p}
	\Mod_{p}\left( \Gamma_{\varepsilon}(h) \right).
\end{equation}
\end{proposition}
\begin{proof}
In the "only if"-direction, Proposition \ref{lemm:localityupper} implies
\begin{equation*}
	\Mod_{p}( \Gamma_{\varepsilon}(h) )
	\leq
	\int_{\Omega} \left( \frac{ \rho_{h} }{ \varepsilon } \right)^{p} \,\d\mm_{\X}
	\quad\text{for every $\varepsilon > 0$.}
\end{equation*}
Hence, $\limsup_{ \varepsilon \rightarrow 0^{+} } \varepsilon^{p} \Mod_{p}( \Gamma_{\varepsilon}(h) ) \leq \int_{\Omega} \rho_{h}^{p} \,\d\mm_{\X}$.
It remains to prove the claim that if \eqref{eq:inequalitygood:measure:proof} holds, then $\UG^{p}( h ) \neq \emptyset$ and
\begin{equation}\label{eq:liminfinequality}
	\int_{\Omega} \rho_{h}^{p} \,\d\mm_{\X}
	\leq
	\liminf_{ \varepsilon \rightarrow 0^{+} }
	\varepsilon^{p}
	\Mod_{p}\left( \Gamma_{\varepsilon}(h) \right).
\end{equation}
For this direction, we refer the reader to the proof of \cite[Theorem 3.10]{Wil:10}. In \cite[Theorem 3.10]{Wil:10} the author  shows that under the assumption of \eqref{eq:inequalitygood:measure:proof} and the further assumption
\begin{equation}\label{eq:integrability}
	\int_{\Omega}
		( {\sf d}_{\rm{Z}}( z_{0}, h(y) ) )^{p}
	\,\d\mm_{\Y}
	<
	\infty
	\quad\text{for some } z_{0} \in \rm{Z},
\end{equation}
the mapping $h$ is an element of a suitable Sobolev space. However, even without the assumption \eqref{eq:integrability} the same proof shows that \eqref{eq:inequalitygood:measure:proof} implies that $\UG^{p}( h ) \neq \emptyset$ and that \eqref{eq:liminfinequality} holds.
\end{proof}
\begin{proposition}\label{prop:equiv_wug}
Let \((\X,\sfd_\X,\mm)\) be a metric measure space and \((\Y,\sfd_\Y)\)
a separable metric space. Let \(h\colon\X\to\Y\) be a Borel map
and \((y_n)_n\subset\Y\) a dense sequence. Set
\(\eta_{n,k}\coloneqq\eta_{y_n,k}\) for all \(n,k\in\N\), where
\(\eta_{y_n,k}\) is given by \eqref{eq:def_good_cut-off}.
Let \(g\in\mathcal L^p_{\rm loc}(\mm)\), \(g\geq 0\) be given.
Then the following are equivalent:
\begin{itemize}
\item[\(\rm i)\)] \(g\) is a weak upper gradient of
\(\eta_{n,k}\circ h\) for every \(n,k\in\N\).
\item[\(\rm ii)\)] \(g\) is a weak upper gradient of \(h\).
\end{itemize}
\end{proposition}
\begin{proof}
\ \\
{\color{blue}\({\rm i)}\Rightarrow{\rm ii)}\)} First, given Proposition \ref{lemm:localityupper}, we
can find a set \(\mathcal N\subset{\rm AC}(\X)\) with
\({\rm Mod}_p(\mathcal N)=0\) such that for every
\(\gamma\in{\rm AC}(\X)\setminus\mathcal N\) and \(n,k\in\N\) it
holds that \(\eta_{n,k}\circ h\circ\gamma\in{\rm AC}(\R)\) and
\(\big|(\eta_{n,k}\circ h\circ\gamma)'\big|(t)
\leq g(\gamma(t))|\gamma'|(t)\) for a.e.\ \(t\in I_\gamma\).
In light of Remark \ref{rmk:finite_on_curves}, we can also require
that \(\int_{ \gamma }g\,\d s<+\infty\)
for every \(\gamma\in{\rm AC}(\X)\setminus\mathcal N\). Therefore,
by applying Lemma \ref{lem:equiv_ms} we deduce that for any
\(\gamma\in{\rm AC}(\X)\setminus\mathcal N\) it holds that
\( h\circ\gamma\) is absolutely continuous and
\(\big|( h\circ\gamma)' \big|(t)\leq g(\gamma(t))|\gamma'|(t)\)
for a.e.\ \(t\in I_\gamma\), so \(g\) is a weak upper gradient
of \( h\).\\
{\color{blue}\({\rm ii)}\Rightarrow{\rm i)}\)} Each function
\(\eta_{n,k}\) is \(1\)-Lipschitz, thus for
\({\rm Mod}_p\)-a.e.\ rectifiable \(\gamma\in{\rm C}(\X)\)
it holds that
\[
\big|(\eta_{n,k}\circ h)(\gamma(b))-
(\eta_{n,k}\circ h)(\gamma(a))\big|\leq
\sfd_\Y\big( h(\gamma(b)), h(\gamma(a))\big)
\leq\int_{ \gamma } g\,\d s,
\]
where \(I_\gamma=[a,b]\). This shows that
\(g\in{\rm UG}^p_{\rm loc}(\eta_{n,k}\circ h)\)
for every \(n,k\in\N\), as desired.
\end{proof}

\subsection{Dirichlet spaces}
We define the \emph{local \(p\)-Dirichlet space} and the
\emph{\(p\)-Dirichlet space} from \(\X\) to \(\Y\) as
\[\begin{split}
\D^{1,p}_\loc(\X,\Y)&\coloneqq\big\{ h\colon\X\to\Y\text{ measurable}\;\big|\;
\UG^p_\loc( h)\neq\emptyset\big\},\\
\D^{1,p}(\X,\Y)&\coloneqq\big\{ h\colon\X\to\Y\text{ measurable}\;\big|\;
\UG^p( h)\neq\emptyset\big\},
\end{split}\]
respectively.
When the target is \(\Y=\R\), we write \(\D^{1,p}(\X)\)
instead of \(\D^{1,p}(\X,\R)\). Calling \({\rm LIP}_{bs}(\X)\)
the set of all Lipschitz functions \(f\colon\X\to\R\) with
bounded support, we have \({\rm LIP}_{bs}(\X)\subset\D^{1,p}(\X)\).

The \emph{\(p\)-Sobolev class} and the \emph{\(p\)-Sobolev
space} over \((\X,\sfd,\mm)\) are given by
\[
\S^{1,p}(\X)\coloneqq\pi_\X\big(\D^{1,p}(\X)\big),\qquad
\W^{1,p}(\X)\coloneqq\S^{1,p}(\X)\cap L^p(\mm),
\]
respectively. Here, $\pi_{\X}$ is the projection mapping identifying each measurable function with its equivalence class up to $\mm$-a.e.\ equality.

\begin{remark}\label{rem:minimalupper}{\rm
Every $h \in \S^{1,p}( \X )$ has a uniquely determined minimal upper gradient, denoted by $|Dh|$, in the following sense: Let $h_{1}, h_{2} \in \D^{1,p}( \X )$, with $h_{1} = h_{2}$ $\mm_{\X}$-almost everywhere, represent $h$. Then $h_{1} - h_{2} \in \D^{1,p}( \X )$ due to Remark \ref{rem:summation}. Lemma \ref{lem:differential} implies that $\rho_{ h_{1} - h_{2} } = 0$ $\mm_{\X}$-almost everywhere, in particular, $\rho_{ h_{1} } = \rho_{ h_{2} }$ $\mm_{\X}$-almost everywhere. Hence we can unambiguously denote $|Dh| = \pi_{\X}( \rho_{ h_{1} } )$.
\fr}\end{remark}

It holds that \(\W^{1,p}(\X)\) is a Banach space
if endowed with the norm
\[
\|f\|_{\W^{1,p}(\X)}\coloneqq\bigg(\int|f|^p\,\d\mm+
\inf_\rho\int\rho^p\,\d\mm\bigg)^{1/p},
\quad\text{ for every }f\in\W^{1,p}(\X),
\]
where the infimum is taken among all those functions
\(\rho\in\mathcal L^p(\mm)\) which are the minimal weak upper
gradient of some representative \(\bar f\in\D^{1,p}(\X)\) of \(f\). We can equivalently express the Sobolev
norm as \(\|f\|_{\W^{1,p}(\X)}=
\big(\int|f|^p\,\d\mm+\int|Df|^p\,\d\mm\big)^{1/p}\)
for every \(f\in\W^{1,p}(\X)\).
\subsection{Outer/inner dilatation and quasiconformality}\label{ss:QC}
We now define the notions of a map with bounded outer/inner dilatation and of a quasiconformal map between metric measure spaces.
\begin{definition}[Outer/inner dilatation and quasiconformality]
Let \((\X,\sfd_\X,\mm_\X)\) and \((\Y,\sfd_\Y,\mm_\Y)\) be metric measure spaces, $1 < p < \infty$, and $\varphi \colon \X \rightarrow \Y$ continuous.
\begin{itemize}
\item[\(\rm i)\)] We say that $\varphi$ has \emph{bounded outer dilatation} provided there exists a constant $K \geq 0$ such that
\begin{equation}\label{eq:K_O_geometric}
	\Mod_{p}( \Gamma ) \leq K \Mod_{p}( \varphi \Gamma ) \quad\text{for all path families $\Gamma$}.
\end{equation}
The smallest such constant \(K\) is denoted by \(K_O(\varphi)\).
\item[\(\rm ii)\)] We say that \(\varphi\) has
\emph{bounded inner dilatation} provided there exists a constant
\(K\geq 0\) such that
\begin{equation}\label{eq:K_I_geometric}
	\Mod_{p}( \varphi\Gamma ) \leq K \Mod_{p}( \Gamma ) \quad\text{for all path families $\Gamma$}.
\end{equation}
The smallest such constant \(K\) is denoted by \(K_I(\varphi)\).
\item[\(\rm iii)\)] We say that a homeomorphism \(\varphi\) is
\emph{\(K\)-quasiconformal} for some \(K\geq 0\) provided it has bounded
outer and inner dilatations and \(K_O(\varphi),K_I(\varphi)\leq K\). The smallest $K$ with this property is called the \emph{maximal dilatation} of $\varphi$.
\end{itemize}
\end{definition}
We also sometimes denote $\varphi_{*}\mm_{\X}( E ) \coloneqq \mm_{\X}( \varphi^{-1}(E) )$ for each Borel $E \subset \Y$.

The following proposition is based on \cite[Theorem 1.1]{Wil:12}, where the equivalence "(1) $\Leftrightarrow$ (3)" is proved under the assumption that $\varphi$ is a homeomorphism. The equivalence "(1) $\Leftrightarrow$ (3)" was recently proved in \cite{Nta:Rom:21} for continuous mappings of bounded outer dilatation, in a different setting.
\begin{proposition}\label{lemm:energydistortion}
Let $\varphi \colon \X \rightarrow \Y$ be a continuous, $p \in (1,\infty)$, and $K \geq 0$. Then the following are equivalent.
\begin{enumerate}
	\item $K_{O}(\varphi) \leq K$;
	\item For every open set $\Omega \subset \Y$ and $\Omega' = \varphi^{-1}(\Omega)$, every metric space $\rm{Z}$, and every $h \in \D^{1,p}( \Omega, \rm{Z}  )$, we have $h \circ \varphi|_{ \Omega' } \in \D^{1,p}( \Omega', \rm{Z} )$ with
	\begin{equation}\label{eq:inequalitygood}
		\int_{\Omega'} ( g \circ \varphi ) \rho_{ h \circ \varphi|_{\Omega'} }^{p} \,\d \mm_{\X} \leq K \int_{\Omega} g \rho_{ h }^{p} \,\d \mm_{\Y}
	\end{equation}
	for every nonnegative $g$ with $g \rho_{ h }^{p} \in \mathcal{L}^{1}( \mm_{\Y}|_{\mathscr{B}( \Omega )} )$.
	\item[(2')] The properties in (2) hold when the word "open" is replaced by "closed" and $\mm_{\Y}( \Omega ) < \infty$.
	\item $\varphi \in \D^{1,p}_{\loc}( \X, \Y )$ and
	\begin{equation}\label{eq:inequalitygood:measure}
		\int_{\X} \1_{E} \circ \varphi \rho_{\varphi}^{p} \,\d\mm_{\X} \leq K \mm_{\Y}( E )\quad\text{for every Borel $E \subset \Y$}.
	\end{equation}
\end{enumerate}
\end{proposition}
\begin{remark}\label{rmk:basic_ineq_phi}{\rm
Suppose that $\varphi \colon \X \rightarrow \Y$ is continuous and satisfies Proposition \ref{lemm:energydistortion} (3). Let \(\rho_\varphi\in\UG^p_\loc(\varphi)\) be a minimal weak upper gradient of \(\varphi\). Then it holds that
\begin{equation}\label{eq:basic_ineq_phi}
h\circ\varphi\,\rho_\varphi\in\mathcal L^p(\mm_\X),\;\;\;
\int_{\X}|h\circ\varphi\,\rho_\varphi|^p\,\d\mm_\X\leq K\int_{\Y}|h|^p\,\d\mm_\Y
\quad\text{ for every }h\in\mathcal L^p(\mm_\Y).
\end{equation}
Indeed, \eqref{eq:basic_ineq_phi} follows from \eqref{eq:inequalitygood:measure} via approximation by simple functions.
\fr}\end{remark}
\begin{proof}[Proof of Proposition \ref{lemm:energydistortion}]
"(1) $\Rightarrow$ (2)": Let $\Omega \subset \Y$ be open, $\Omega' = \varphi^{-1}(\Omega)$ and $h \in \D^{1,p}( \Omega, \rm{Z} )$. We apply Proposition \ref{lem:dirichlet} to the mapping $h|_{\Omega} \circ \varphi|_{\Omega'}$. Assumption (1) and $h \in \D^{1,p}( \Omega, \rm{Z} )$ imply $h \circ \varphi|_{ \Omega' } \in \D^{1,p}( \Omega', \rm{Z} )$, with
\begin{equation*}
	\int_{\Omega'}
		\rho_{ h \circ \varphi|_{ \Omega' } }^{p}
	\,\d\mm_{\X}
	\leq
	K\int_{\Omega}
		\rho_{ h|_{\Omega} }^{p}
	\,\d\mm_{\Y}
\end{equation*}
following from \eqref{eq:actuallimit}. Given that $h|_{U} \in \D^{1,p}( \Omega, \rm{Z} )$ for every open $U \subset \Omega$, the locality of the minimal weak upper gradients, Lemma \ref{lemm:locality:minimal}, shows
\begin{equation*}
	\int_{ \varphi^{-1}( U ) }
		\rho_{ h \circ \varphi|_{\Omega'} }^{p}
	\,\d\mm_{\X}
	\leq
	K\int_{U}
		\rho_{h}^{p}
	\,\d\mm_{\Y}
	\quad\text{for every open $U \subset \Omega$.}
\end{equation*}
Consequently, the measures $\nu(E) = \int_{ \varphi^{-1}(E) } \rho_{h \circ \varphi|_{\Omega'}}^{p} \,\d\mm_{\X}$ and $\mu(E) = K \int_{E} \rho_{h}^{p} \,\d\mm_{\Y}$ are Borel regular and finite measures satisfying $\nu( E ) \leq \mu( E )$ whenever $E \subset \Omega$ is open. The outer regularity of the measures implies the validity of $\nu(E) \leq \mu(E)$ for each Borel $E \subset \Omega$. This inequality implies  \eqref{eq:inequalitygood}. Indeed, given $g \rho_{ h }^{p} \in \mathcal{L}^{1}( \mm_{\Y}|_{\mathscr{B}( \Omega )} )$, there exists a pointwise increasing sequence of nonnegative simple functions $g_{n}$ such that $g_n \rho_{h}^{p}$ converges to $g \rho_{h}^{p}$ in $\mathcal{L}^{1}( \mm_{\Y}|_{\mathscr{B}( \Omega )} )$. The inequality $\nu \leq \mu$ and dominated convergence imply
\begin{equation*}
	\int_{ \Omega' } ( g \circ \varphi ) \rho_{ h \circ \varphi }^{p} \,\d\mm_{\X}
	=
	\int_{ \Omega }
		g
	\,\d\nu
	=
	\lim_{ n \rightarrow \infty }
	\int_{ \Omega } g_n \,\d\nu
	\leq
	\lim_{ n \rightarrow \infty }
	\int_{ \Omega } g_n \,\d\mu
	=
	\int_{ \Omega } g \,\d\mu.
\end{equation*}
This establishes \eqref{eq:inequalitygood}.

"(1) $\Rightarrow$ (2')": Let $F \subset \Y$ be closed and $F' = \varphi^{-1}(F)$. We equip $F'$ with the restricted measure $\mm_{\X}|_{ \mathscr{B}( F' ) }$. Then $\varphi|_{F'}$ satisfies (1) with the same constant $K$. Thus (2) applies to $\varphi|_{F'}$ and the set $\Omega = F$, which is relatively open in $F$. Hence (2') follows.

"(2) $\Rightarrow$ (3)": Let $\Omega \subset \Y$ satisfy $\mm_{\Y}( \Omega ) < \infty$ and $U = \varphi^{-1}( \Omega )$. We apply \eqref{eq:inequalitygood} to $h = \mathrm{id}_{\Omega}$, which implies \eqref{eq:inequalitygood:measure} for such $\Omega$. This also implies $\varphi|_{U} \in \D^{1,p}( U, \Y )$. Since $\X$ can be covered by such $U$, Lemma \ref{lemm:locality:dirichlet} yields $\varphi \in \D^{1,p}_{\loc}( \X, \Y )$. Clearly \eqref{eq:inequalitygood:measure} holds if $\mm_{\Y}( \Omega ) = \infty$ and $\Omega$ is open. The claim for arbitrary Borel sets follows from outer regularity.

"(2') $\Rightarrow$ (3)": Consider an open set $\Omega \subset \Y$, $U = \varphi^{-1}(\Omega)$, and an open set $\Omega_{1} \subset \Y$ for which $F_{1} \coloneqq \overline{ \Omega }_{1} \subset \Omega$ and $\mm_{\Y}( F_1 ) < \infty$. Denote $U_{1} = \varphi^{-1}( F_{1} )$. By arguing as in "(2) $\Rightarrow$ (3)", we conclude $\varphi|_{ U_{1} } \in \D^{1,p}( U_{1}, \Y )$ and
\begin{equation}\label{eq:intermediatestep}
	\int_{ \varphi^{-1}(\Omega_1) }
		\rho_{ \varphi|_{ \varphi^{-1}(\Omega_1) } }^{p}
	\,\d\mm_{\X}
	\leq
	\int_{ \varphi^{-1}(\Omega_1) }
		\rho_{ \varphi|_{ U_1 } }^{p}
	\,\d\mm_{\X}
	\leq
	K \mm_{\Y}( F_1 )
	\leq
	K \mm_{\Y}( \Omega ).
\end{equation}
Now consider a nested sequence of open sets $\Omega_{i} \subset \Y$ for which $\Omega_{i} \subset F_{i} \coloneqq \overline{ \Omega }_{i} \subset \Omega$ with $\mm_{\Y}( F_{i} ) < \infty$ and $\Omega = \bigcup_{ i = 1 }^{ \infty } \Omega_{i}$. Since $U = \bigcup_{ i = 1 }^{ \infty } \varphi^{-1}( \Omega_{i} )$, we conclude from Lemma \ref{lemm:locality:dirichlet} that $\varphi|_{U} \in \D^{1,p}_{\loc}( U, \Y )$. The locality of the minimal weak upper gradients and \eqref{eq:intermediatestep} imply
\begin{equation*}
	\int_{ U }
		\rho_{ \varphi|_{ U } }^{p}
	\,\d\mm_{\X}
	=
	\lim_{ i \rightarrow \infty }
	\int_{ \varphi^{-1}(\Omega_i) }
		\rho_{ \varphi|_{ \varphi^{-1}(\Omega_i) } }^{p}
	\,\d\mm_{\X}
	\leq
	K \mm_{\Y}( \Omega ).
\end{equation*}
By the arbitrariness of $\Omega$, we have $\varphi \in \D^{1,p}_{\loc}( \X, \Y )$, so $\rho_{ \varphi|_{ U } } = ( \rho_{ \varphi } )|_{U}$ $\mm_{\X}$-almost everywhere. Hence \eqref{eq:inequalitygood:measure} holds for open sets, and outer regularity yields the rest.

"(3) $\Rightarrow$ (1)": Fix a representative $\rho_{ \varphi }$. Let $\mathcal{N}$ denote the collection of absolutely continuous $\gamma \in \AC( \X )$ for which either $\varphi \circ \gamma$ is not absolutely continuous, $\int_{ \gamma } \rho_{  \varphi } \,\d s = \infty$, or $| (\varphi \circ \gamma)' |(t) \leq \rho_{ \varphi }( \gamma(t) ) | \gamma' |(t)$ does not hold almost everywhere in $I_{\gamma}$. We recall from Proposition \ref{lemm:localityupper} that $\Mod_{p}( \mathcal{N} ) = 0$. Let $\mathcal{N}'$ denote the collection of all rectifiable $\gamma \in \C( \X )$ for which the unit speed parametrization $\widetilde{\gamma}$ of $\gamma$ is an element of $\mathcal{N}$. We have $\Mod_{p}( \mathcal{N}' ) = 0$ since $\Mod_{p}( \mathcal{N} ) = 0$.

Let $\Gamma$ be a path family. If $\Mod_{p}( \varphi\Gamma ) = \infty$, we are done. Hence, we may assume $\Mod_{p}( \varphi\Gamma ) < \infty$. Consider an admissible function $\rho \in \mathcal{L}^{p}( \mm_{\Y} )$ for the path family $\varphi \Gamma$. Denote $\rho' = ( \rho \circ \varphi ) \rho_{\varphi}$. Then, for every $\gamma \in \Gamma \setminus \mathcal{N}'$, $1 \leq \int_{ \varphi \circ \gamma } \rho \,\d s \leq \int_{ \gamma } \rho'\,\d s.$ Hence, $\rho'$ is admissible for the path family $\Gamma \setminus \mathcal{N}'$, and we conclude from Remark \ref{rmk:basic_ineq_phi} that
\begin{equation*}
	\Mod_{p}( \Gamma )
	=
	\Mod_{p}( \Gamma \setminus \mathcal{N}' )
	\leq
	\int_{ \X } ( \rho' )^{p} \,\d\mm_{\X}
	\leq
	K \int_{\Y} \rho^{p} \,\d\mm_{\Y}.
\end{equation*}
Taking the infimum over such admissible $\rho$ shows (1).
\end{proof}
\subsection{Cotangent module}\label{ss:cotg_mod}
We recall the notion of \(L^p\)-normed \(L^\infty\)-module
introduced in \cite{Gig:18}:
\begin{definition}[Normed module]
Let \((\X,\sfd,\mm)\) be a metric measure space and let \(p\in(1,\infty)\).
Let \(\mathscr M\) be an algebraic module over the commutative ring \(L^\infty(\mm)\).
Then a \emph{pointwise norm} on \(\mathscr M\) is any map
\(|\cdot|\colon\mathscr M\to L^p(\mm)\) such that the following conditions
are verified (in the \(\mm\)-a.e.\ sense):
\begin{equation}\label{eq:pointwisenorm}
\begin{split}
|v|\geq 0,&\quad\text{ for every }v\in\mathscr M,
\text{ with equality if and only if }v=0,\\
|hv|=|h||v|,&\quad\text{ for every }v\in\mathscr M\text{ and }h\in L^\infty(\mm),\\
|v+w|\leq|v|+|w|,&\quad\text{ for every }v,w\in\mathscr M.
\end{split}
\end{equation}
We say that \(\mathscr M\) is an \emph{\(L^p(\mm)\)-normed \(L^\infty(\mm)\)-module}
provided it is a Banach space when endowed with the norm induced by \(|\cdot|\),
which is defined as
\[
{\|v\|}_{\mathscr M}\coloneqq\bigg(\int|v|^p\,\d\mm\bigg)^{1/p},
\quad\text{ for every }v\in\mathscr M.
\]
\end{definition}
A fundamental example of normed module is the cotangent module
over a metric measure space, which we are going to recall below.
This object was introduced by Gigli in \cite{Gig:18}, in the
case \(p=2\). The easy adaptation to arbitrary exponents
\(p\in(1,\infty)\) was obtained in \cite{Gig:Pas:19}.

We recall from Remark \ref{rem:minimalupper} that every equivalence class $f \in\S^{1,p}( \X )$ of Dirichlet functions has a uniquely determined equivalence class of minimal weak upper gradients $|Df| \in L^{p}( \mm_{\X} )$.
\begin{theorem}[Cotangent module]
Let \((\X,\sfd,\mm)\) be a metric measure space and let
\(p\in(1,\infty)\). Then there exists a unique pair \(\big(L^p(T^*\X),\d\big)\)
-- where \(L^p(T^*\X)\) is an \(L^p(\mm)\)-normed
\(L^\infty(\mm)\)-module and \(\d\colon\S^{1,p}(\X)\to L^p(T^*\X)\)
is a linear operator -- such that:
\begin{itemize}
\item[\(\rm i)\)] \(|\d f|=|Df|\) holds \(\mm\)-a.e.\ in \(\X\)
for every \(f\in\S^{1,p}(\X)\).
\item[\(\rm ii)\)] The set \(\big\{\sum_{i=1}^n h_i\,\d f_i\;\big|\;n\in\N,
\;(h_i)_{i=1}^n\subset L^\infty(\mm),\;(f_i)_{i=1}^n\subset\S^{1,p}(\X)\big\}\)
is dense in \(L^p(T^*\X)\).
\end{itemize}
We call \(L^p(T^*\X)\) the \emph{cotangent module} of \((\X,\sfd,\mm)\)
and \(\d\) the associated \emph{differential}.

Uniqueness is intended up to unique isomorphism, meaning that
if another pair \((\mathscr M,T)\) verifies the same properties,
then there exists a unique pointwise norm preserving, linear
bijection \(\Phi\colon L^p(T^*\X)\to\mathscr M\) such that
\(T=\Phi\circ\d\).
\end{theorem}
\begin{remark}{\rm
In \cite{Gig:Pas:19} the differential $\d$ was only defined on $\S^{1,p}( \X ) \cap L^{p}_{\loc}( \mm_{\X} )$. However, if $u \in \D^{1,p}( \X )$ represents an equivalence class of $f \in \S^{1,p}( \X )$, observe that $u_{k} = \max\left\{-k, \min\left\{u,k\right\} \right\} \in \D^{1,p}( \X ) \cap L^{p}_{\loc}( \mm_{\X} )$ for every $k \in \mathbb{N}$. If we set $f_{k} = \pi_{\X}( u_{k} )$, the differential $\d f_{k}$ is well-defined and
\begin{equation*}
	| \d f_{k} - \d f_{n} |
	=
	\rho_{ u_k - u_n }
	\leq
	\rho_{ u_{k} - u } + \rho_{ u - u_{n} }
	\quad\text{$\mm_{\X}$-almost everywhere}.
\end{equation*}
As $k \rightarrow \infty$, $\rho_{ u_{k} - u } \rightarrow 0$ in $L^{p}( \mm_{\X} )$ due to \cite[Proposition 7.1.18]{HKST:15}. Hence $( \d f_k )_{ k = 1 }^{ \infty }$ is a Cauchy sequence with some limit $\omega \in L^{p}( T^{*}\X )$. It follows from Lemma \ref{lem:differential} that $\1_{ \left\{ |u| \leq k \right\} } \omega = \d f_{k}$ for every $k \in \mathbb{N}$, hence $| D f | = | \omega |$, and we denote $\d f = \omega$.
\fr}\end{remark}
Occasionally, given a function \(f\in\D^{1,p}(\X)\), we will
write \(\d f\in L^p(T^*\X)\) in place of \(\d\big(\pi_\X(f)\big)\).
\begin{remark}\label{rmk:implic_iH}{\rm
It follows from the proof of \cite[Proposition 2.2.10]{Gig:18}
that if \(L^p(T^*\X)\) is reflexive, then \(\W^{1,p}(\X)\) is
reflexive. The validity of the converse
implication is not known.
Moreover, there do exist metric measure spaces
(see \cite{Amb:Col:DiMa:15}) whose Sobolev space
\(\W^{1,p}(\X)\) is not reflexive, and thus a fortiori
\(L^p(T^*\X)\) is not reflexive as well.
\fr}\end{remark}

The \emph{dual} of a \(L^p(\mm)\)-normed \(L^\infty(\mm)\)-module \(\mathscr M\)
is the \(L^q(\mm)\)-normed \(L^\infty(\mm)\)-module \(\mathscr M^*\) (where
\(\frac{1}{p}+\frac{1}{q}=1\)) which is defined as follows: \(\mathscr M^*\) is the space
of all \(L^\infty(\mm)\)-linear continuous maps \(T\colon\mathscr M\to L^1(\mm)\)
endowed with the pointwise norm
\begin{equation}\label{eq:def_dual_ptwse_norm}
|T|\coloneqq\underset{\substack{v\in\mathscr M:\\|v|\leq 1\,\mm\text{-a.e.}}}
{\rm ess\,sup}\big|\langle T,v\rangle\big|,\quad\text{ for every }T\in\mathscr M^*.
\end{equation}
The \emph{tangent module} of \((\X,\sfd,\mm)\) is then defined as
\(L^q(T\X)\coloneqq L^p(T^*\X)^*\). However, in order to keep a notational consistency
with the classical Riemannian framework, the duality pairing between \(v\in L^p(T^*\X)\)
and \(\omega\in L^q(T\X)\) will be denoted by \(\langle\omega,v\rangle\) in place of
\(\langle v,\omega\rangle\).
\medskip

In Theorem \ref{thm:def_diff} we will also deal with another notion of normed module, where
the integrability assumption is removed. Given a metric measure space \((\X,\sfd,\mm)\), we
denote by \(L^0(\mm)\) the space of all real-valued Borel functions on \(\X\) considered up
to \(\mm\)-a.e.\ equality. It holds that \(L^0(\mm)\) is a topological vector space when endowed with the
distance \(\sfd_{L^0(\mm)}(f,g)\coloneqq\int\min\{|f-g|,1\}\,\d\tilde\mm\),
where \(\tilde\mm\) is any Borel probability measure on \(\X\) such that
\(\mm\ll\tilde\mm\ll\mm\). Following \cite[Definition 2.6]{Gig:17}, we say that
an algebraic module \(\mathscr M^0\) over \(L^0(\mm)\) is an \emph{\(L^0(\mm)\)-normed
\(L^0(\mm)\)-module} provided it is endowed with a pointwise norm
\(|\cdot|\colon\mathscr M^0\to L^0(\mm)\) satisfying the axioms \eqref{eq:pointwisenorm}
for every \(v,w\in\mathscr M^0\), \(h\in L^0(\mm)\) and such that the distance
\(\sfd_{\mathscr M^0}(v,w)\coloneqq\sfd_{L^0(\mm)}(|v-w|,0)\) is complete.
\medskip

As proved in \cite[Theorem/Definition 1.7]{Gig:17}, every \(L^p(\mm)\)-normed
\(L^\infty(\mm)\)-module can be canonically `extended' to a \(L^0(\mm)\)-normed
\(L^0(\mm)\)-module, as we now recall. Given an \(L^p(\mm)\)-normed
\(L^\infty(\mm)\)-module \(\mathscr M\), there exists a unique pair
\((\mathscr M^0,\iota)\), where \(\mathscr M^0\) is an \(L^0(\mm)\)-normed
\(L^0(\mm)\)-module (called the \emph{\(L^0\)-completion} of \(\mathscr M\))
and \(\iota\colon\mathscr M\to\mathscr M^0\) is a linear map that preserves the
pointwise norm and has dense image. Then we define \(L^0(T^*\X)\) and \(L^0(T\X)\)
as the \(L^0\)-completions of \(L^p(T^*\X)\) and \(L^q(T\X)\), respectively.
Notice that \(L^0(T^*\X)\) and \(L^0(T\X)\) depend on the fixed exponent \(p\),
even though for the sake of brevity such dependence is omitted from our notation.
\medskip

The dual \((\mathscr M^0)^*\) of an \(L^0(\mm)\)-normed \(L^0(\mm)\)-module \(\mathscr M\)
is the space of \(L^0(\mm)\)-linear continuous maps \(T\colon\mathscr M^0\to L^0(\mm)\),
endowed with the natural pointwise norm defined in analogy with
\eqref{eq:def_dual_ptwse_norm}. Another useful construction is the following (we refer
to \cite[last paragraph of Section 2]{Gig:Pas:Sou:20} for the details). Given any set
\(E\in\mathscr B(\X)\), we have a natural \emph{extension} operator
\({\rm ext}\colon L^0(\mm|_E)\to L^0(\mm)\), sending each \(f\in L^0(\mm|_E)\) to the
function coinciding with \(f\) \(\mm\)-a.e.\ on \(E\), with \(0\) \(\mm\)-a.e.\ on
\(\X\setminus E\). A similar procedure allows to extend `by zero' every
\(L^0(\mm|_E)\)-normed \(L^0(\mm|_E)\)-module \(\mathscr M^0\) to an \(L^0(\mm)\)-normed
\(L^0(\mm)\)-module \({\rm Ext}(\mathscr M^0)\). We denote by
\({\rm ext}\colon\mathscr M^0\to{\rm Ext}(\mathscr M^0)\) the associated extension operator.
It is easy to verify that the dual of \({\rm Ext}(\mathscr M^0)\) can be identified with
\({\rm Ext}(\mathscr N^0)\), where \(\mathscr N^0\coloneqq(\mathscr M^0)^*\).
Finally, we need the following notion of \emph{pullback} of a \(L^0\)-normed \(L^0\)-module:
\begin{theorem}[Pullback of a normed module]\label{thm:pullback_NMod}
Let \((\X,\sfd_\X,\mm_\X)\), \((\Y,\sfd_\Y,\mm_\Y)\) be metric measure spaces. Let
\(\varphi\colon\X\to\Y\) be a Borel measurable map such that \(\varphi_*\mm_\X\ll\mm_\Y\).
Let \(\mathscr M^0\) be an \(L^0(\mm_\Y)\)-normed \(L^0(\mm_\Y)\)-module. Then there exists
a unique pair \(\big(\varphi^*\mathscr M^0,[\varphi^*\#]\big)\), where
\(\varphi^*\mathscr M^0\) is an \(L^0(\mm_\X)\)-normed \(L^0(\mm_\X)\)-module, while
\([\varphi^*\#]\colon\mathscr M^0\to\varphi^*\mathscr M^0\) is a linear operator such that:
\begin{itemize}
\item[\(\rm i)\)] \(\big|[\varphi^*v]\big|=|v|\circ\varphi\) holds \(\mm_\X\)-a.e.\ for
every \(v\in\mathscr M^0\).
\item[\(\rm ii)\)] The \(L^0(\mm_\X)\)-linear combinations of the elements of
\(\big\{[\varphi^*v]\,:\,v\in\mathscr M^0\big\}\) are dense in \(\varphi^*\mathscr M^0\).
\end{itemize}
\end{theorem}
Observe that in item i) of Theorem \ref{thm:pullback_NMod} the element
\(|v|\circ\varphi\in L^0(\mm_\X)\) is well-posed thanks to the assumption
\(\varphi_*\mm_\X\ll\mm_\Y\), which ensures that if two given Borel functions
\(f,g\colon\X\to\R\) agree \(\mm_\Y\)-a.e., then \(f\circ\varphi=g\circ\varphi\)
holds \(\mm_\X\)-a.e..
\medskip

Finally, we recall some necessary terminology needed in the Subsection \ref{sec:necessary_cond}, and refer the interested reader \cite[Section 1.4]{Gig:18} 
for further reading on this topic.

Let \((\X,\sfd,\mm)\) be a metric measure space and \(\mathscr M\) an \(L^p(\mm)\)-normed \(L^\infty(\mm)\)-module
for some \(p\in(1,\infty)\). Then we say that elements $\omega_{1}, \omega_{2}, \dots, \omega_{k} \in\mathscr M$ are \emph{independent} in a Borel set $E \subset \X$ if
\begin{equation*}
	\1_{E} \sum_{ i = 1 }^{ k } f_{i}\omega_{i} = 0 \in\mathscr M
\end{equation*}
for $f_{1}, f_2, \dots, f_{k} \in L^{\infty}( \mm )$ only if $f_{1}, f_{2}, \dots, f_{k} = 0$ $\mm$-almost everywhere in $E$.

Let $E \subset \X$ be Borel and $\omega_{1}, \omega_{2}, \dots, \omega_{k} \in \mathscr M$, for finite $k = 1, 2, \dots$. The \emph{span in $E$} of $\omega_{1}, \omega_{2}, \dots, \omega_{k}$, denoted by $\textrm{Span}_{E}( \omega_1, \omega_2, \dots, \omega_{k} )$ is the collection of all $\omega \in \mathscr M$ such that
\begin{equation*}
\omega = \sum_{ i = 1 }^{ \infty } \1_{ E_{i} } f_{i} \omega_{ j_i }
\end{equation*}
for some Borel partition $( E_{i} )_{ i = 1 }^{ \infty }$ of $E$, $( f_{i} )_{ i = 1 }^{ \infty } \subset L^{\infty}( \mm )$, and $j_{i} \in \left\{1,2,\dots,k\right\}$ for each $i \in \mathbb{N}$.

We say that $\omega_{1}, \omega_{2}, \dots, \omega_{k} \in\mathscr M$ form a \emph{basis} for \(\mathscr M\) in $E$ if they are independent in $E$ and $\1_{E}\omega \in \textrm{Span}_{E}( \omega_1, \omega_2, \dots, \omega_{k} )$ for every $\omega \in \mathscr M$. In this case, we say that $\1_{E}\mathscr M$ is \emph{$k$-dimensional}. If $\1_{E}\omega = 0$ for every $\omega \in \mathscr M$, we say that \(\1_{E}\mathscr M\) is \emph{$0$-dimensional}. Lastly, we say that $\1_{E}\mathscr M$ is \emph{$\infty$-dimensional} if \(\1_F\mathscr M\) is not $k$-dimensional for any finite $k$ and any \(F\subset E\) with \(\mm(F)>0\).

We recall the existence of a \emph{dimensional decomposition} $(D_{i})_{ i= 1 }^{ \infty }$ of $\mathscr M$: this means that $( D_i )_{ i = 0 }^{ \infty }$ is a Borel decomposition of $\X$ and \(\1_{D_i}\mathscr M\) is $i$-dimensional, for every $i$. Such a decomposition always exists \cite[Proposition 1.4.5]{Gig:18}.
\section{Main results}\label{sec:main}
\subsection{Pullback of maps with bounded outer dilatation}
\label{ss:pullback}
The aim of this subsection is to construct a pullback operator
associated with \(\varphi\), acting at the level of cotangent modules.
Before doing so, we prove some auxiliary results concerning the
properties of the pre-composition with \(\varphi\).
\medskip

Hereafter, our standing assumptions are the following:
\begin{itemize}
\item
\((\X,\sfd_\X,\mm_\X)\) and \((\Y,\sfd_\Y,\mm_\Y)\)
are metric measure spaces,
\item \(p\in(1,\infty)\) is a fixed exponent,
\item \(\varphi\colon\X\to\Y\) is a mapping with bounded outer dilatation with $K_{O}( \varphi ) \leq K$.
\end{itemize}
\begin{lemma}\label{cor:Dirichlet}
Let $\varphi \colon \X \rightarrow \Y$ be a mapping with bounded outer dilatation and $p \in (1,\infty)$. If $h_{1}, h_{2} \in \D^{1,p}( \Y )$, then $h_{1} \circ \varphi, h_{2} \circ \varphi, ( h_{1} - h_{2} ) \circ \varphi \in \D^{1,p}( \X )$. Moreover, if $E \subset \left\{ h_{1} = h_{2} \right\} \cup \left\{ \rho_{ h_1 - h_{2} } = 0 \right\} \cup N_{0}$ is a Borel set, where $\mm_{\Y}( N_{0} ) = 0$, then $\rho_{ h_{1} \circ \varphi - h_{2} \circ \varphi } = 0$ $\mm_{\X}$-almost everywhere in $\varphi^{-1}(E)$.
\end{lemma}
\begin{proof}
First, let $h_{1}, h_{2} \in \D^{1,p}( \Y )$. We recall from Remark \ref{rem:summation} that $( h_1 - h_2 ) \in \D^{1,p}( \Y )$. Proposition \ref{lemm:energydistortion} (2) yields that 
$h_{1} \circ \varphi, h_{2} \circ \varphi, ( h_{1} - h_{2}) \circ \varphi \in \D^{1,p}( \X )$. Lemma \ref{lem:differential} establishes that $\rho_{ h_{1} - h_{2} } = 0$ $\mm_{\Y}$-almost everywhere in $E$. Then Proposition \ref{lemm:energydistortion} (2) yields that $\rho_{ ( h_{1} - h_{2} ) \circ \varphi } \equiv 0$ $\mm_{\X}$-almost everywhere in $\varphi^{-1}(E)$.
\end{proof}
Using Lemma \ref{cor:Dirichlet}, we observe that for each $f \in \S^{1,p}( \Y )$ and each representative $\bar f \in \D^{1,p}( \Y )$ of $f$, we have $\bar f \circ \varphi \in \D^{1,p}( \X )$. Observe also that if $h \in \D^{1,p}( \Y )$ is another representative of $f$, then the same Lemma establishes the equality $\rho_{ \bar f \circ \varphi - h \circ \varphi } = 0$ $\mm_{\X}$-almost everywhere. Hence, $\d( \pi_{\X}( \bar f \circ \varphi ) ) = \d( \pi_{\X}( h \circ \varphi ) ) \in L^{p}( T^{*}\X )$. We pose the following definition.
\begin{definition}
Fix \(f\in\S^{1,p}(\Y)\) and a representative $\bar f \in \D^{1,p}( \Y )$ of $f$. Then we define the element
\(\d(f\circ\varphi)\in L^p(T^*\X)\) as
\begin{equation}\label{eq:def_d_f_comp_phi}
\d(f\circ\varphi)\coloneqq\d\big(\pi_\X(\bar f\circ\varphi)\big).
\end{equation}
\end{definition}
Before stating the definition \eqref{eq:def_d_f_comp_phi}, we argued that $\d\big(\pi_\X(\bar f\circ\varphi)\big)$ is independent of the chosen representative $\bar f \in \D^{1,p}( \Y )$ of $f$. Hence, the notation $\d( f \circ \varphi )$ is justified.
\begin{lemma}\label{lemm:pullbackgradient}
Let $\varphi \colon \X \rightarrow \Y$ be a mapping with bounded dilatation, $h \in \D^{1,p}( \Y, \rm{Z} )$ and $g \in \UG^{p}( h )$. Then $\widehat{g} = ( g \circ \varphi ) \rho_{\varphi} \in \UG^{p}( h \circ \varphi )$.
\end{lemma}
\begin{proof}
For $\Mod_{p}$-almost every $\theta \in \AC( \Y )$, $h \circ \theta$ is absolutely continuous, $\int_{ \theta } g \,\d s = \infty$, and $| ( h \circ \theta )' |(t) \leq ( g(\theta) )(s) |\theta'|(s)$ almost everywhere in $I_{\theta}$. Let $\mathcal{N}_{h,g} \subset \AC( \Y )$ denote the collection of paths for which some of these properties fail. We recall from Proposition \ref{lemm:localityupper} that $\Mod_{p}( \mathcal{N}_{h,g} ) = 0$.

Proposition \ref{lemm:energydistortion} (3) states that $\varphi \in \D^{1,p}_{\loc}( \X, \Y )$, so we may fix a representative $\rho_{\varphi}$ for $|D \varphi|$. Let $\mathcal{N}_{ \varphi, \rho_{\varphi} } \subset \AC( \X )$ be the path family defined analogously to $\mathcal{N}_{h,g}$. Reapplying Proposition \ref{lemm:localityupper} shows $\Mod_{p}( \mathcal{N}_{\varphi, \rho_{\varphi}} ) = 0$.

Let $\mathcal{N}_{0} \subset \C( \X )$ denote the collection of all rectifiable paths whose unit speed parametrizations $\widetilde{\gamma}$ are contained in $\mathcal{N}_{ \varphi, \rho_{\varphi} }$ or that satisfy $\varphi \circ \widetilde{\gamma} \in \mathcal{N}_{ h, g }$. It follows from Proposition \ref{lemm:energydistortion} (1) that $\Mod_{p}( \mathcal{N}_{0} ) = 0$. For every $\gamma \in \AC( \X ) \setminus \mathcal{N}_{0}$, we have $\gamma \in \AC( \X ) \setminus \mathcal{N}_{ \varphi, \rho_{\varphi} }$ and $\varphi \circ \gamma \in \AC( \Y ) \setminus \mathcal{N}_{h,g}$, hence $( h \circ \varphi ) \circ \gamma$ is absolutely continuous, $\widehat{g}$ is integrable over $\gamma$, and
\begin{equation*}
	| ( h \circ \varphi \circ \gamma )' |(t) \leq \widehat{g}( \gamma(t) ) | \gamma' |(t) \quad\text{almost everywhere in $I_{\gamma}$.}
\end{equation*}
It follows from \eqref{eq:basic_ineq_phi} that $\widehat{g} \in \mathcal{L}^{p}( \mm_{\X} )$, so $\widehat{g} \in \UG^{p}( h \circ \varphi )$.
\end{proof}
\begin{theorem}[Pullback]\label{thm:pullback_map}
Let \((\X,\sfd_\X,\mm_\X)\), \((\Y,\sfd_\Y,\mm_\Y)\)
be metric measure spaces, $p\in (1,\infty)$, and $K > 0$. Let \(\varphi\colon\X\to\Y\) be mapping with bounded outer dilatation, with $K_{O}( \varphi ) \leq K$.
Then there exists a unique linear, continuous operator
\(\varphi^*\colon L^p(T^*\Y)\to L^p(T^*\X)\),
called the \emph{pullback map} of \(\varphi\), such that
\begin{equation}\label{eq:def_pullback}
\varphi^*(h\,\d f)=\pi_\X(\bar h\circ\varphi)\,\d(f\circ\varphi),
\quad\text{ for every }f\in\S^{1,p}(\Y)\text{ and }h\in L^\infty(\mm_\Y),
\end{equation}
where \(\bar h\colon\Y\to\R\) is any bounded Borel function such that
\(h=\pi_\Y(\bar h)\). Moreover, it holds that
\begin{equation}\label{eq:ineq_pullback}
|\varphi^*\omega|\leq|D\varphi|\,\pi_\X\big(\overline{|\omega|}\circ\varphi\big)
\;\;\;\mm_\X\text{-a.e.\ in }\X,\quad\begin{array}{ll}
\text{ for every }\omega\in L^p(T^*\Y)\text{ and every}\\
\text{ representative }\overline{|\omega|}\in\mathcal L^p(\mm_\Y)\text{ of }|\omega|.
\end{array}
\end{equation}
In particular,
\begin{equation}\label{eq:int_ineq_pullback}
{\|\varphi^*\omega\|}_{L^p(T^*\X)}\leq K^{1/p}{\|\omega\|}_{L^p(T^*\Y)},
\quad\text{ for every }\omega\in L^p(T^*\Y).
\end{equation}
\end{theorem}
\begin{proof}
First of all, we observe that the right hand side of \eqref{eq:def_pullback} is
well-defined: given two representatives \(\bar h,\bar h' \in \D^{1,p}( \Y ) \) of \(h\),
we have that \(N\coloneqq\{\bar h\neq\bar h'\}\) is \(\mm_\Y\)-negligible, 
thus \(\1_{\varphi^{-1}(N)}^{\mm_\X}\d(f\circ\varphi)=0\) by Lemma \ref{cor:Dirichlet}; this implies that
\(\pi_\X(\bar h\circ\varphi)\,\d(f\circ\varphi)=\pi_\X(\bar h'\circ\varphi)\,\d(f\circ\varphi)\). Similarly, the right hand side of \eqref{eq:ineq_pullback} does
not depend on the chosen representative \(\overline{|\omega|}\) of \(|\omega|\).

Let us now define the linear space \(\mathcal V\subset L^p(T^*\Y)\) as
\[
\mathcal V\coloneqq\bigg\{\sum_{i=1}^n\1_{E_i}^{\mm_\Y}\d f_i\;\bigg|\;
n\in\N,\;(E_i)_{i=1}^n\text{ Borel partition of }\Y,
\;(f_i)_{i=1}^n\subset\S^{1,p}(\Y)\bigg\}.
\]
Observe that the space \(\mathcal V\) is dense in \(L^p(T^*\X)\).
By linearity and \eqref{eq:def_pullback}, we define the
pullback map \(\varphi^*\colon\mathcal V\to L^p(T^*\X)\) as
\begin{equation}\label{eq:def_pullback_on_simple_forms}
\varphi^*\omega\coloneqq\sum_{i=1}^n\1_{\varphi^{-1}(E_i)}^{\mm_\X}\d(f_i\circ\varphi),
\quad\text{ for every }\omega=\sum_{i=1}^n\1_{E_i}^{\mm_\Y}\d f_i\in\mathcal V.
\end{equation}
To prove that the above definition is well-posed (namely, that \(\varphi^*\omega\)
is independent of the particular way of representing \(\omega\)), let us fix a
representative \(\overline{|Df_i|}\in\mathcal L^p(\mm_\Y)\) of \(|Df_i|\)
for any \(i=1,\ldots,n\). Then
\[\begin{split}
\bigg|\sum_{i=1}^n\1_{\varphi^{-1}(E_i)}^{\mm_\X}\d(f_i\circ\varphi)\bigg|&=\sum_{i=1}^n\1_{\varphi^{-1}(E_i)}^{\mm_\X}\big|\d(f_i\circ\varphi)\big|
\leq|D\varphi|\sum_{i=1}^n\1_{\varphi^{-1}(E_i)}^{\mm_\X}\,
\pi_\X\big(\overline{|Df_i|}\circ\varphi\big)\\
&=|D\varphi|\,\pi_\X\bigg(\bigg(\sum_{i=1}^n
\1_{E_i}\,\overline{|Df_i|}\bigg)\circ\varphi\bigg)
\end{split}\]
is satisfied \(\mm_\X\)-a.e.\ on \(\X\) by Lemma \ref{lemm:pullbackgradient}.
Since the function \(\sum_{i=1}^n\1_{E_i}\,\overline{|Df_i|}\in\mathcal L^p(\mm_\Y)\)
is a representative of \(|\omega|\), we deduce that the definition in 
\eqref{eq:def_pullback_on_simple_forms} is well-posed and that all the elements
of \(\mathcal V\) satisfy \eqref{eq:ineq_pullback}. Proposition \ref{lemm:energydistortion} (2) implies that all the elements of \(\mathcal V\)
satisfy the inequality in \eqref{eq:int_ineq_pullback}. This grants
that the map \(\varphi^*\) -- which is linear by construction -- is continuous.
Therefore, it can be uniquely extended to a linear and continuous map
\(\varphi^*\colon L^p(T^*\Y)\to L^p(T^*\X)\).

To prove that \(\varphi^*\)
satisfies \eqref{eq:def_pullback}, let us fix any \(h\in L^\infty(\mm_\Y)\)
and \(f\in\S^{1,p}(\Y)\). Choose any sequence \((h_n)_n\subset{\sf Sf}(\Y)\)
that converges to \(h\) in \(L^\infty(\mm_\Y)\). Then we can select some
bounded Borel functions \(\bar h_n,\bar h\colon\Y\to\R\) such that
\(\pi_\Y(\bar h_n)=h_n\) for all \(n\in\N\), \(\pi_\Y(\bar h)=h\)
and \(\bar h_n\to\bar h\) uniformly on \(\Y\). Thanks to
\eqref{eq:def_pullback_on_simple_forms} and the linearity of \(\varphi^*\), we know that
\begin{equation}\label{eq:def_pullback_aux}
\varphi^*(h_n\d f)=\pi_\X(\bar h_n\circ\varphi)\,\d(f\circ\varphi),
\quad\text{ for every }n\in\N.
\end{equation}
Notice that \(\bar h_n\circ\varphi\to\bar h\circ\varphi\) uniformly on \(\X\),
thus \(\pi_\X(\bar h_n\circ\varphi)\to\pi_\X(\bar h\circ\varphi)\) in
\(L^\infty(\mm_\X)\). Hence, by letting \(n\to\infty\) in
\eqref{eq:def_pullback_aux} we conclude that
\(\varphi^*(h\,\d f)=\pi_\X(\bar h\circ\varphi)\,\d(f\circ\varphi)\).
This proves \eqref{eq:def_pullback}.

Let us now prove that \eqref{eq:ineq_pullback} is verified for any
\(\omega\in L^p(T^*\Y)\). Choose a sequence \((\omega_n)_n\subset\mathcal V\)
such that \(\omega_n\to\omega\) in \(L^p(T^*\Y)\). Since \(\varphi^*\) is continuous,
it also holds that \(\varphi^*\omega_n\to\varphi^*\omega\) in \(L^p(T^*\X)\).
Up to passing to a subsequence and relabeling, we have that \(|\omega_n|\to|\omega|\)
\(\mm_\Y\)-a.e.\ and \(|\varphi^*\omega_n|\to|\varphi^*\omega|\) \(\mm_\X\)-a.e.
Then there exist representatives
\(\overline{|\omega_n|},\overline{|\omega|}\in\mathcal L^p(\mm_\Y)\)
of \(|\omega_n|,|\omega|\) such that
\(\overline{|\omega_n|}(y)\to\overline{|\omega|}(y)\) for every \(y\in\Y\).
We know that all the elements of \(\mathcal V\) satisfy \eqref{eq:ineq_pullback},
thus in particular
\begin{equation}\label{eq:ineq_pullback_aux}
|\varphi^*\omega_n|\leq|D\varphi|\,\pi_\X\big(\overline{|\omega_n|}\circ\varphi\big)
\;\;\;\mm_\X\text{-a.e.\ in }\X,\quad\text{ for every }n\in\N.
\end{equation}
Notice that \(\big(\overline{|\omega_n|}\circ\varphi\big)(x)\to
\big(\overline{|\omega|}\circ\varphi\big)(x)\) for every \(x\in\X\),
so that by letting \(n\to\infty\) in \eqref{eq:ineq_pullback_aux} we conclude that
\(|\varphi^*\omega|\leq|D\varphi|\,\pi_\X\big(\overline{|\omega|}\circ\varphi\big)\)
holds \(\mm_\X\)-a.e.\ in \(\X\). This proves \eqref{eq:ineq_pullback}.
Finally, the inequality in \eqref{eq:int_ineq_pullback} follows
from \eqref{eq:ineq_pullback} by Remark \ref{rmk:basic_ineq_phi}.
Therefore, the statement is achieved.
\end{proof}
\subsection{Consistency of the pullback}\label{sec:consistency}
Under the assumptions of Theorem \ref{thm:pullback_map}, we have the pullback operator $\varphi^{*} \colon L^{p}( T^{*}\Y ) \to L^{p}( T^{*}\X )$ provided by Theorem \ref{thm:pullback_map}. The \emph{adjoint} of the pullback operator is the unique $\mathbb{R}$-linear map $\varphi_{*} \colon L^{q}( T\X ) \to L^{q}( T\Y )$ determined by the following action:
\begin{equation*}
	L_{ \varphi_{*}v }( \omega )
	\coloneqq
	\int_{\X}
		\langle \varphi^{*}(\omega), v \rangle
	\,\d\mm_{\X}
	\quad\text{for all $\omega \in L^{p}( T^{*}\Y )$}.
\end{equation*}
To see that $\varphi_{*}v$ is well-defined, Theorem \ref{thm:pullback_map} implies the $\mathbb{R}$-linearity and that $$| L_{ \varphi_{*}v }( \omega ) | \leq K^{1/p}\| v \|_{ L^{q}( T\X ) } \| \omega \|_{ L^{p}( T^{*}\Y ) } \quad\text{for every $\omega \in L^{p}( T^{*}\Y )$.}$$ Hence there exists a unique $\varphi_{*}v \in L^{q}( T\Y )$ satisfying
\begin{equation*}
	L_{ \varphi_{*}v }( \omega )
	=
	\int_{\Y}
		\langle \omega, \varphi_{*}v \rangle
	\,\d\mm_{\Y}
	\quad\text{for all $\omega \in L^{p}( T^{*}\Y )$},
\end{equation*}
with $\| \varphi_{*}v \|_{ L^{q}( T^{*}\Y ) } \leq K^{1/p} \| v \|_{ L^{q}( T\X ) }$ \cite[Proposition 1.2.13]{Gig:18}.

A second way to define the pullback (and thus its adjoint) is using the mapping \[[\varphi^*\#]\colon L^0(T^*\Y)\to\varphi^*L^0(T^*\Y)\] from Theorem
\ref{thm:pullback_NMod}, the target \(\varphi^*L^0(T^*\Y)\) being intended as an
\(L^0(\tilde\mm_\X)\)-normed \(L^0(\tilde\mm_\X)\)-module, where we set
\(\tilde\mm_\X\coloneqq\mm_\X|_{\{\rho_\varphi>0\}}\).
Indeed, the measures \(\varphi_*\tilde\mm_\X\) and \(\varphi_*(\rho_\varphi^p\mm_\X)\)
have the same negligible sets and \(\varphi_*(\rho_\varphi^p\mm_\X)\leq K\mm_\Y\),
thus in particular \(\varphi_*\tilde\mm_\X\ll\mm_\Y\) and accordingly Theorem
\ref{thm:pullback_NMod} can be applied.

By mimicking the construction of differential in
\cite[Section 3]{Gig:Pas:Sou:20}, which was in turn inspired by the notion of differential
of a map of bounded deformation introduced in \cite[Proposition 2.4.6]{Gig:18}, we obtain
the following:
\begin{theorem}[Differential]\label{thm:def_diff}
Let \((\X,\sfd_\X,\mm_\X)\), \((\Y,\sfd_\Y,\mm_\Y)\), and \(\varphi\) be as in
Theorem \ref{thm:pullback_map}. Then there exists a unique \(L^0(\mm_\X)\)-linear continuous
map \(\d\varphi\colon L^0(T\X)\to{\rm Ext}\big(\varphi^*L^0(T^*\Y)\big)^*\) such that
\[
\big\langle{\rm ext}([\varphi^*\d f]),\d\varphi(v)\big\rangle=
\big\langle\d(f\circ\varphi),v\big\rangle,\quad\text{ for every }
v\in L^0(T\X)\text{ and }f\in\S^{1,p}(\Y).
\]
Moreover, for every \(v\in L^0(T\X)\) it holds that
\begin{equation}\label{eq:ineq_diff}
\big|\d\varphi(v)\big|\leq|D\varphi||v|,\quad\mm_\X\text{-a.e.\ on }\X.
\end{equation}
\end{theorem}
\begin{proof}
Given any \(v\in L^0(T\X)\) and \(\omega\in L^p(T^*\Y)\), we define
\(T_v([\varphi^*\omega])\coloneqq\langle\varphi^*\omega,v\rangle\in L^0(\mm_\X)\). Since
\[
\big|T_v([\varphi^*\omega])\big|=\big|\langle\varphi^*\omega,v\rangle\big|
\leq|\varphi^*\omega||v|\overset{\eqref{eq:ineq_pullback}}\leq
|D\varphi|\,\pi_\X\big(\overline{|\omega|}\circ\varphi\big)|v|
=|D\varphi|\,{\rm ext}\big(\big|[\varphi^*\omega]\big|\big)|v|,\quad\mm_\X\text{-a.e.},
\]
we deduce that the map \(T_v\) is well-defined and can be uniquely extended
(by arguing as in \cite[Proposition 2.16]{Gig:17}) to an element
\(\d\varphi(v)\in{\rm Ext}\big(\varphi^*L^0(T^*\Y)\big)^*\) satisfying \eqref{eq:ineq_diff}.
Given that \(\d\varphi\) is \(L^0(\mm_\X)\)-linear operator by construction, we infer from
\eqref{eq:ineq_diff} that \(\d\varphi\) is also continuous. Finally, its uniqueness is
granted by the fact that \(L^p(T^*\Y)\) is generated by
\(\big\{\d f\,:\,f\in\S^{1,p}(\Y)\big\}\).
\end{proof}
We see from Theorem \ref{thm:def_diff} that the differential $\d \varphi$ and $\varphi_{*}$ are related in the following way:
\begin{equation*}
	L_{ \varphi_{*}v }( \omega )
	=
	\int_{\X}
		\langle{\rm ext}([\varphi^*\omega]),\d\varphi(v)\big\rangle
	\,\d\mm_{\X}
	\quad\text{for every $v \in L^{q}(T\X)$ and every $\omega \in L^{p}( T^{*}\Y )$.}
\end{equation*}
We consider the consistency between Theorem \ref{thm:def_diff} and
\cite[Definition 3.4]{Gig:Pas:Sou:20}. In \cite{Gig:Pas:Sou:20} the differential
\(\hat\d\varphi\) is defined (under more restrictive assumptions on \(\varphi\) and
for \(p=2\)) as an operator of the form
\[
\hat\d\varphi\colon L^0(T\X)\to{\rm Ext}\big(\varphi^*L^0_\mu(T^*\Y)\big)^*,
\]
where we set \(\mu\coloneqq\varphi_*(\rho_\varphi^p\mm_\X)\) and \(L^0_\mu(T^*\Y)\)
stands for the cotangent module of \((\Y,\sfd_\Y,\mu)\).

More specifically, the differential \(\hat\d\varphi\) is characterised
by the identity
\begin{equation}\label{eq:def_hat_d_varphi}
\big\langle{\rm ext}([\varphi^*\d_\mu f]),\hat\d\varphi(v)\big\rangle
=\langle\d g,v\rangle
\end{equation}
for every \(v\in L^0(T\X)\), \(f\in\S^{1,p}(\Y,\mu)\), and
\(g\in\S^{1,p}(\X)\) such that \(g=f\circ\varphi\)
\(\mm_\X\)-a.e.\ on \(\{\rho_\varphi>0\}\); the existence of
such \(g\) and the fact that \(\1_{\{\rho_\varphi>0\}}\d g\) is
independent of the specific \(g\) were proved in
\cite[Proposition 3.3]{Gig:Pas:Sou:20}. Recall that
\(\mu=\varphi_*(\rho_\varphi^p\mm_\X)\leq K\mm_\Y\), as pointed out
in Proposition \ref{lemm:energydistortion}. This ensures that
\(\S^{1,p}(\Y)=\S^{1,p}(\Y,\mm_\Y)\subset\S^{1,p}(\Y,\mu)\)
and \(|Df|_\mu\leq|Df|_{\mm_\Y}=|Df|\) holds \(\mu\)-a.e.\ for
every \(f\in\S^{1,p}(\Y)\). In turn, this implies that there exists
a unique \(L^0(\mm_\X)\)-linear continuous operator
\(P\colon{\rm Ext}\big(\varphi^*L^0(T^*\Y)\big)\to{\rm Ext}
\big(\varphi^*L^0_\mu(T^*\Y)\big)\) satisfying
\[
P\big({\rm ext}([\varphi^*\d f])\big)=
{\rm ext}([\varphi^*\d_\mu f]),\quad\text{ for every }f\in\S^{1,p}(\Y).
\]
Call \(P_*\colon{\rm Ext}\big(\varphi^*L^0_\mu(T^*\Y)\big)^*\to{\rm Ext}
\big(\varphi^*L^0(T^*\Y)\big)^*\) the adjoint operator of \(P\).
Then we claim that
\begin{equation}\label{eq:claim_consist_diff}
\d\varphi=P_*\circ\hat\d\varphi.
\end{equation}
To prove \eqref{eq:claim_consist_diff}, observe that if a function \(f\in\S^{1,p}(\Y)\subset\S^{1,p}(\Y,\mu)\)
is given, then it is possible to choose \(g\coloneqq f\circ\varphi\in\S^{1,p}(\X)\)
in \eqref{eq:def_hat_d_varphi}. Consequently, for any \(f\in\S^{1,p}(\Y)\) and
\(v\in L^0(T\X)\) we have
\[\begin{split}
\big\langle{\rm ext}([\varphi^*\d f]),P_*\big(\hat\d\varphi(v)\big)\big\rangle
&=\big\langle P\big({\rm ext}([\varphi^*\d f])\big),\hat\d\varphi(v)\big\rangle
=\big\langle{\rm ext}([\varphi^*\d_\mu f]),\hat\d\varphi(v)\big\rangle\\
&=\big\langle\d(f\circ\varphi),v\big\rangle
=\big\langle{\rm ext}([\varphi^*\d f]),\d\varphi(v)\big\rangle.
\end{split}\]
This shows that \(\d\varphi=P_*\circ\hat\d\varphi\), thus giving the claim \eqref{eq:claim_consist_diff}.
\subsection{Necessary conditions for quasiconformality}\label{sec:necessary_cond}
\label{ss:nec_QC}
We aim to show that the reflexivity of the
cotangent module is invariant under quasiconformal maps. In other
words, a necessary condition for two spaces \(\X\)
and \(\Y\) to be quasiconformally equivalent is that either both
\(L^p(T^*\X)\) and \(L^p(T^*\Y)\) are reflexive, or that neither
of them is reflexive. This result will be stated in Theorem
\ref{thm:qc_vs_refl}.
\medskip

Given that the family of maps with bounded outer dilatation is
closed under composition, we can easily demonstrate the expected
functorial nature of the associated pullback operator:
\begin{lemma}[Functoriality]\label{lem:functorial}
Let \((\X,\sfd_\X,\mm_\X)\), \((\Y,\sfd_\Y,\mm_\Y)\), and
\(({\rm Z},\sfd_{\rm Z},\mm_{\rm Z})\) be metric measure spaces.
Let \(\varphi\colon\X\to\Y\) and \(\psi\colon\Y\to{\rm Z}\) be continuous mappings satisfying $K_{O}( \varphi ) \leq K$ and $K_{O}( \varphi' ) \leq K'$. Then
\((\psi\circ\varphi)^*=\varphi^*\circ\psi^*\) and $K_{O}( \psi \circ \varphi ) \leq K K'$.
\end{lemma}
\begin{proof}
Given any set \(\Gamma\subset\C(\X)\), it holds that
\(\psi\varphi\Gamma=(\psi\circ\varphi)\Gamma\) and thus $K_{O}( \psi \circ \varphi ) \leq K K'$ follows.

The operator \(\varphi^*\circ\psi^*\colon L^p(T^*{\rm Z})
\to L^p(T^*\X)\) is linear and continuous, as a composition of linear
and continuous maps. Given any \(f\in\S^{1,p}({\rm Z})\) and
\(h\in L^\infty(\mm_{\rm Z})\), with representatives
\(\bar f\in\D^{1,p}({\rm Z})\) and
\(\bar h\in\mathcal L^\infty(\mm_{\rm Z})\), respectively, we
deduce from the definitions of \(\varphi^*\) and \(\psi^*\) that
\[\begin{split}
(\varphi^*\circ\psi^*)(h\,\d f)&=
\varphi^*\big(\pi_\Y(\bar h\circ\psi)\,\d(f\circ\psi)\big)=
\pi_\X\big(\bar h\circ(\psi\circ\varphi)\big)\,
\d\big(f\circ(\psi\circ\varphi)\big).
\end{split}\]
Thanks to the arbitrariness of \(f\in\S^{1,p}({\rm Z})\)
and \(h\in L^\infty(\mm_{\rm Z})\), we conclude that
\((\psi\circ\varphi)^*=\varphi^*\circ\psi^*\).
\end{proof}

By applying Lemma \ref{lem:functorial} to a quasiconformal map
and its inverse, we can immediately deduce that the pullback of a
quasiconformal map is an isomorphism of Banach spaces:
\begin{corollary}\label{cor:pullback_QC}
Let \((\X,\sfd_\X,\mm_\X)\) and \((\Y,\sfd_\Y,\mm_\Y)\) be two
metric measure spaces. Let \(\varphi\colon\X\to\Y\) be a
\(K\)-quasiconformal map. Denote \(\psi\coloneqq\varphi^{-1}\).
Then the pullback maps \(\varphi^*\colon L^p(T^*\Y)\to L^p(T^*\X)\)
and \(\psi^*\colon L^p(T^*\X)\to L^p(T^*\Y)\) are isomorphisms
of Banach spaces, one the inverse of the other. In particular, it holds that
\begin{equation}\label{eq:int_ineq_qc}
\frac{1}{K}\int|\omega|^p\,\d\mm_\Y\leq\int|\varphi^*\omega|^p\,\d\mm_\X\leq
K\int|\omega|^p\,\d\mm_\Y,\quad\text{ for every }\omega\in L^p(T^*\Y).
\end{equation}
\end{corollary}
\begin{proof}
First, note that \({\rm id}_\X^*={\rm id}_{L^p(T^*\X)}\)
and \({\rm id}_\Y^*={\rm id}_{L^p(T^*\Y)}\). Hence, combining
Lemma \ref{lem:functorial} with the fact that \(\varphi\circ\psi
={\rm id}_\Y\) and \(\psi\circ\varphi={\rm id}_\X\), we get
\(\psi^*=(\varphi^*)^{-1}\). Moreover, the second inequality in
\eqref{eq:int_ineq_qc} is exactly \eqref{eq:int_ineq_pullback}
applied to \(\varphi\), while the first one follows from
\eqref{eq:int_ineq_pullback} applied to \(\psi\):
\[
\int|\omega|^p\,\d\mm_\Y=\int\big|\psi^*(\varphi^*\omega)\big|^p\,\d\mm_\Y
\leq K\int|\varphi^*\omega|^p\,\d\mm_\X,
\quad\text{ for every }\omega\in L^p(T^*\Y).
\]
This completes the proof of the statement.
\end{proof}
\begin{theorem}[Quasiconformality and reflexivity]
\label{thm:qc_vs_refl}
Let \((\X,\sfd_\X,\mm_\X)\) and \((\Y,\sfd_\Y,\mm_\Y)\) be two metric
measure spaces with \(L^p(T^*\X)\) reflexive. Suppose there exists
a quasiconformal map \(\varphi\colon\X\to\Y\). Then \(L^p(T^*\Y)\)
is reflexive.
\end{theorem}
\begin{proof}
This immediately follows from Corollary \ref{cor:pullback_QC},
taking into account the fact that reflexivity of Banach
spaces is preserved under the passage to an equivalent norm.
\end{proof}
\begin{theorem}[Quasiconformality and dimensional decomposition]
\label{thm:qc_vs_dimension}
Let \((\X,\sfd_\X,\mm_\X)\) and \((\Y,\sfd_\Y,\mm_\Y)\) be metric
measure spaces and $( D_{i} )_{ i = 0 }^{ \infty }$ a dimensional decomposition of $L^{p}( T^{*}\Y )$.

Suppose there exists a quasiconformal homeomorphism \(\varphi\colon\X\to\Y\). Then there exists a Borel set $E_{0} \supset D_{0}$ with $\mm_{\Y}( E_{0} \setminus D_0 ) = 0$ such that $D_{0}' = \varphi^{-1}( E_0 )$, $D_{i}' = \varphi^{-1}( D_i \setminus E_0 )$ defines a dimensional decomposition of $L^{p}( T^{*}\X )$, with $D_{i}'$ having dimension $i$ for each $i = 0, 1, \dots, \infty$, and that $\varphi|_{ \X \setminus E_{0} }$ satisfies Lusin's Conditions ($N$) and ($N^{-1}$).
\end{theorem}

\begin{remark}{\rm
Theorem \ref{thm:qc_vs_dimension} is sharp in the following sense: Sometimes the $E_{0}$ in Theorem \ref{thm:qc_vs_dimension} is strictly larger than $D_0$. To see why, let $\X$ denote the Euclidean plane $\mathbb{R}^{2}$ endowed with the Lebesgue measure. We consider the dimensional decomposition $D_2 = \mathbb{R}^{2}$ for $L^{2}( T^{*}\X )$. Now, there exists a metric measure space $\Y$ and a quasiconformal homeomorphism $\varphi \colon \mathbb{R}^{2} \rightarrow \Y$ onto a metric measure space $\Y$, where $L^{2}( T^{*}\Y )$ has a dimensional decomposition $D_{0}', D_{2}'$ with $\mm_{\Y}( D_{0}' ) > 0$. In fact, we may take $\Y \subset \mathbb{R}^{3}$ such that $D_{0}'$ is a (self-similar) purely $2$-unrectifiable Cantor set and $D_{2}' = \Y \setminus D_{0}'$ a smooth surface, with $\mm_{\Y}$ being the $2$-dimensional Hausdorff measure restricted to $\Y$; such a $\Y$ appears in \cite[Proposition 17.1]{Raj:17}. A weaker version of Theorem \ref{thm:qc_vs_dimension} was observed in \cite[Lemma 3.2 and Remark 3.4]{Iko:21:unif} by the first named author.
\fr}\end{remark}
\begin{lemma}\label{lemma:zerodimensionality}
Let $\varphi \colon \X \rightarrow \Y$ be a quasiconformal homeomorphism. Then there exists a Borel set $E_1 \subset \Y$ such that
\begin{itemize}
	\item $\1_{E_1}L^{p}( T^{*}\Y )$ and $\1_{\varphi^{-1}(E_1)}L^{p}( T^{*}\X )$ are zero-dimensional.
	\item If $F \subset \Y$ is zero-dimensional, then $\mm_{\Y}( F \setminus E_1 ) = 0$ and $F' \subset \X$ zero-dimensional, then $\mm_{\X}( F' \setminus \varphi^{-1}( E_1 ) ) = 0$.
	\item $\varphi|_{ \X \setminus \varphi^{-1}( E_1 ) }$ satisfies Lusin's Condition ($N$) and ($N^{-1}$).
\end{itemize}
\end{lemma}
\begin{proof}
Fix a representative $\rho_{ \mathrm{id}_{\Y} }$ of $|D\mathrm{id}_{\Y}|$ and $E_{\Y} = \left\{ \rho_{ \mathrm{id}_{\Y} } = 0 \right\}$. It follows from Lemma \ref{lemm:pullbackgradient} that $\rho_{ h } = 0$ $\mm_{\Y}$-almost everywhere in $E_{\Y}$ for every $h \in \D^{1,p}( \Y )$. Hence every Borel subset of $E_{\Y}$ is zero-dimensional. Conversely, if $E \subset \Y$ is zero-dimensional, then $\rho_{h} = 0$ $\mm_{\Y}$-almost everywhere in $E$ for every $h \in \D^{1,p}( \Y )$. We claim that $\mm_{\Y}( E \setminus E_{\Y} ) = 0$. To see why, let $( y_n )_{ n = 1 }^{ \infty } \subset \Y$ be dense and the $\eta_{n,k}$ the $1$-Lipschitz mappings from Proposition \ref{prop:equiv_wug}. Then $\rho_{ \mathrm{id}_{\Y} } \geq \rho_{ \eta_{n,k} } \eqqcolon g_{n,k}$ $\mm_{\Y}$-almost everywhere. Also, $g \coloneqq \sup_{k} \sup_{n} g_{n,k} \leq \rho_{ \mathrm{id}_{\Y} }$ $\mm_{\Y}$-almost everywhere. It follows from Proposition \ref{prop:equiv_wug} that $g \in \UG_{\loc}^{p}( \mathrm{id}_{\Y} )$. Hence $g = \rho_{ \mathrm{id}_{\Y} }$ $\mm_{\Y}$-almost everywhere. Since $g$ vanishes $\mm_{\Y}$-almost everywhere in $E$, we conclude that $\mm_{\Y}( E \setminus E_{\Y} ) = 0$.

Define $E_{\X}$ correspondingly and set $E = \varphi( E_{\X} ) \cup E_{\Y}$. It follows from Corollary \ref{cor:pullback_QC} that $\varphi( E_{\X} )$ is zero-dimensional. Hence $\mm_{\Y}( E \setminus E_{\Y} ) = 0$. By symmetry, $\varphi^{-1}( E )$ is zero-dimensional and $\mm_{\X}( \varphi^{-1}(E) \setminus E_{\X} ) = 0$.

It follows from Proposition \ref{lemm:energydistortion} (2) that
\begin{equation}
	\label{eq:distortioninequality}
	K^{-1}
	\int_{F}
		\rho_{ \mathrm{id}_{\Y} }^{p}
	\,\d\mm_{\Y}
	\leq
	\int_{ \varphi^{-1}(F) }
		\rho_{ \varphi }^{p}
	\,\d\mm_{\X}
	\leq
	K
	\int_{F}
		\rho_{ \mathrm{id}_{\Y} }^{p}
	\,\d\mm_{\Y}
	\quad\text{for every Borel $F \subset \Y$}.
\end{equation}
Indeed, the first inequality follows by applying Proposition \ref{lemm:energydistortion} (2) to $h = \varphi$ and its inverse. The second inequality follows by applying Proposition \ref{lemm:energydistortion} (2) to $h = \mathrm{id}_{\Y}$ and $\varphi$.

We conclude from \eqref{eq:distortioninequality} that $\left\{ \rho_{\varphi} = 0 \right\} \setminus \varphi^{-1}(E)$ and $\varphi^{-1}( E ) \setminus \left\{ \rho_{\varphi} = 0 \right\}$ have zero $\mm_{\X}$-measure. Hence $\1_{ \left\{ \rho_{\varphi} > 0 \right\} }\mm_{\X}$ and $\1_{ \X \setminus \varphi^{-1}(E) }\mm_{\X}$ have the same negligible sets, which is also true for the measures $\1_{ \left\{ \rho_{ \varphi^{-1}} > 0 \right\} }\mm_{\Y}$ and $\1_{ \Y \setminus E }\mm_{\Y}$. With these facts at hand, Proposition \ref{lemm:energydistortion} (3) implies that $\varphi|_{ \X \setminus \varphi^{-1}(E) }$ satisfies Lusin's Condition ($N$) and ($N^{-1}$). Our claim follows by setting $E_1 = E$.
\end{proof}

\begin{lemma}\label{lemma:dimensionpreservation}
Let $\varphi \colon \X \rightarrow \Y$ be a quasiconformal homeomorphism and $E_{1} \subset \Y$ a Borel set satisfying the conclusion of Lemma \ref{lemma:zerodimensionality}, and $k = 0, 1, \dots$, finite. If $E \subset \Y \setminus E_{1}$ is Borel, then $\1_{E}L^{p}( T^{*}\Y )$ is $k$-dimensional if and only if $\1_{ \varphi^{-1}(E) }L^{p}( T^{*}\X )$ is $k$-dimensional.
\end{lemma}
\begin{proof}
Let $E_{1} \subset \Y$ be a Borel set satisfying the conclusions of Lemma \ref{lemma:zerodimensionality}. Since $\varphi|_{ \X \setminus \varphi^{-1}(E_1) }$ satisfies Lusin's Condition ($N$) and ($N^{-1}$), it is immediate from Corollary \ref{cor:pullback_QC} and the definition of $\varphi^{*}$ that $\varphi^{*}$ restricted to $\1_{ E \setminus E_1 }L^{p}( T^{*}\Y )$ maps onto $\1_{ \varphi^{-1}( E \setminus E_1 ) }L^{p}( T^{*}\X )$, while preserving the linear independence and spans. This yields the claim.
\end{proof}
\begin{proof}[Proof of Theorem \ref{thm:qc_vs_dimension}]
Let $( D_{i} )_{ i = 0 }^{ \infty }$ be a dimensional decomposition of $\Y$. Let $E_{1} \subset \Y$ be a Borel set satisfying the conclusions of Lemma \ref{lemma:zerodimensionality}, and denote $E_{0} = E_{1} \cup D_{0}$. Here $E_{0}$ also satisfies the conclusions of Lemma \ref{lemma:zerodimensionality} and $E_{0} \supset D_{0}$. Lemma \ref{lemma:dimensionpreservation} yields that every Borel $F_{i}' \subset D_i' = \varphi^{-1}( D_i \setminus E_0 )$ with $\mm_{\X}( F_{i}' ) > 0$ is $i$-dimensional. It follows from Lemma \ref{lemma:zerodimensionality} that $D_{0}' = \varphi^{-1}( E_0 )$ is zero-dimensional. The claim follows from these properties.
\end{proof}
\subsection{Sufficient conditions for quasiconformality}
\label{ss:suf_QC}
In this subsection we address this problem: under which
additional conditions a given map \(\varphi\) with bounded
outer dilatation is quasiconformal? We show that the
invertibility of the pullback map \(\varphi^*\) is a necessary
but not a sufficient condition. The additional requirement one
needs is that pre-composing with \(\varphi^{-1}\) preserves
Sobolev regularity of Lipschitz maps.
\begin{theorem}\label{thm:equiv_QC}
Let \((\X,\sfd_\X,\mm_\X)\) and \((\Y,\sfd_\Y,\mm_\Y)\) be metric
measure spaces. Let \(\varphi\colon\X\to\Y\) be a homeomorphism having
bounded outer dilatation. Then the following conditions are equivalent:
\begin{itemize}
\item[\(\rm i)\)] \(\varphi\) has bounded inner dilatation (and thus
\(\varphi\) is quasiconformal),
\item[\(\rm ii)\)] \(\varphi^*\colon L^p(T^*\Y)\to L^p(T^*\X)\)
is invertible and \(f\circ\varphi^{-1}\in{\rm D}^{1,p}(\Y)\)
for every \(f\in{\rm LIP}_{bs}(\X)\).
\end{itemize}
Moreover, if \(\rm i)\) and \(\rm ii)\) hold, then
\(K_I(\varphi)=C^p\), where \(C\) stands for the operator
norm of \((\varphi^*)^{-1}\).
\end{theorem}
\begin{proof}
\ \\
{\color{blue}\({\rm i)}\Rightarrow{\rm ii)}\)} This follows immediately from Corollary \ref{cor:pullback_QC} and
Proposition \ref{lemm:energydistortion}.\\
{\color{blue}\({\rm ii)}\Rightarrow{\rm i)}\)} Thanks to the Bounded
Inverse Theorem, we know that there exists a constant \(C>0\) such that
\(\|\varphi^*\omega\|_{L^p(T^*\X)}\geq C^{-1/p}\|\omega\|_{L^p(T^*\Y)}\)
for every \(\omega\in L^p(T^*\Y)\). In particular, it holds that
\begin{equation}\label{eq:invertibility_aux1}
\big\|\1_B^{\mm_Y}\d(f\circ\varphi^{-1})\big\|_{L^p(T^*\Y)}\leq
C^{1/p}\big\|\1_{\varphi^{-1}(B)}^{\mm_\X}\d f\big\|_{L^p(T^*\X)},
\quad\forall f\in{\rm LIP}_{bs}(\X),\,B\in\mathscr B(\Y).
\end{equation}
Fix a minimal weak upper gradient \(\rho_\varphi\) of \(\varphi\).
Define \(E\coloneqq\{\rho_\varphi>0\}\in\mathscr B(\X)\)
and \(F\coloneqq\varphi(E)\in\mathscr B(\Y)\).
Fix any \((y_n)_n\subset\Y\) dense and denote
\(\eta_{n,k}\coloneqq\eta_{y_n,k}\in{\rm LIP}_{bs}(\Y)\) for all
\(n,k\in\N\), where \(\eta_{y_n,k}\) is given by
\eqref{eq:def_good_cut-off}. Define
\(g_m\coloneqq\sup_{n,k\leq m}\big|\d(\eta_{n,k}
\circ\varphi^{-1})\big|\) for every \(m\in\N\) and
\(g\coloneqq\sup_{n,k\in\N}\big|\d(\eta_{n,k}\circ\varphi^{-1})\big|\).
We can find Borel partitions \((G^m_{n,k})_{n,k\leq m}\) of \(\Y\)
such that \(g_m=\sum_{n,k\leq m}\1_{G^m_{n,k}}^{\mm_\Y}\big|
\d(\eta_{n,k}\circ\varphi^{-1})\big|\). Then
\[\begin{split}
\int_B g_m^p\,\d\mm_\Y&\overset{\phantom{(\star)}}=
\sum_{n,k\leq m}\big\|\1_{G^m_{n,k}\cap B}^{\mm_\Y}
\,\d(\eta_{n,k}\circ\varphi^{-1})\big\|_{L^p(T^*\Y)}^p
\overset{\eqref{eq:invertibility_aux1}}\leq
C\sum_{n,k\leq m}\big\|\1_{\varphi^{-1}(G^m_{n,k}\cap B)}^{\mm_\X}
\,\d\eta_{n,k}\big\|_{L^p(T^*\X)}^p\\
&\overset{(\star)}\leq C\sum_{n,k\leq m}
\mm_\X\big(\varphi^{-1}(G^m_{n,k})\cap\varphi^{-1}(B)\big)
=C\,\varphi_*\mm_\X(B),\quad\text{ for every }B\in\mathscr B(\Y),
\end{split}\]
where the starred inequality follows from the fact that each
function \(\eta_{n,k}\) is \(1\)-Lipschitz. Given that
\(g_m\nearrow g\) in the \(\mm_\Y\)-a.e.\ sense, an application
of the Monotone Convergence Theorem yields
\begin{equation}\label{eq:invertibility_aux2}
\int_B g^p\,\d\mm_\Y\leq C\,\varphi_*\mm_\X(B)= C \mm_{\X}( \varphi^{-1}(B) )
\quad\text{ for every }B\in\mathscr B(\Y).
\end{equation}
Since \(\mm_\X\) is boundedly finite and \(\varphi^{-1}\) is
continuous, we have that \(\varphi_*\mm_\X\) is locally finite,
thus \eqref{eq:invertibility_aux2} ensures that
\(g\in\mathcal L^p_{\rm loc}(\mm_\Y)\). Hence, we are in a position
to apply Proposition \ref{prop:equiv_wug}: any Borel representative
\(\bar g\) of \(g\) is a weak upper gradient of each
function \(\eta_{n,k}\circ\varphi\), thus
Proposition \ref{prop:equiv_wug} implies that \(\bar g\) is a weak
upper gradient of \(\varphi^{-1}\) and accordingly
\(\varphi^{-1}\in{\rm D}^{1,p}_{\rm loc}(\Y,\X)\).
We have $\rho_{\varphi^{-1}} \leq g$ $\mm_{\Y}$-almost everywhere hence \eqref{eq:invertibility_aux2} implies that $\varphi^{-1}$ satisfies Proposition \ref{lemm:energydistortion} (3) for the constant $K = C$. Hence $K_{O}( \varphi^{-1} ) = K_{I}( \varphi ) \leq C$, where $C^{1/p}$ is the operator norm of $( \varphi^{*} )^{-1}$. The inequality $C \leq K_{O}( \varphi^{-1} ) = K_{I}( \varphi )$ follows from \eqref{eq:int_ineq_pullback} since $( \varphi^{*} )^{-1} = ( \varphi^{-1} )^{*}$. The equality $C = K_{I}( \varphi )$ follows, thereby proving the claim.
\end{proof}
\begin{remark}\label{rem:failure}{\rm
In item ii) of Theorem \ref{thm:equiv_QC}, the requirement that
\(f\circ\varphi^{-1}\in\D^{1,p}(\Y)\) if \(f\in\LIP_{bs}(\X)\)
cannot be disposed of, as \cite[Example 6.1]{Iko:Rom:ar} shows: First, denote $\X \coloneqq ( \mathbb{R}^{2}, \| \cdot \|_{2}, \mathcal{L}^{2} )$. The example of \cite{Iko:Rom:ar} consists of
a metric measure space $\Y \coloneqq ( \Y, \sfd_{\Y}, \mathcal{H}^{2}_{\Y} )$, with $( \Y, \sfd_{\Y} )$ geodesic and $\mathcal{H}^{2}_{\Y}$ the \(2\)-dimensional
Hausdorff measure, for which there exists a $1$-Lipschitz homeomorphism
$\varphi \colon \X \rightarrow \Y$ that satisfies
$K_{O}( \varphi ) = 1$ (for $p = 2$), yet $\varphi$ is not
quasiconformal. The mapping $\varphi$ has the further property that
there exists a Cantor set $E \subset\mathbb{R}\times\left\{0\right\}$
of positive length, $\varphi|_{ \X \setminus E }$
is a local isometry, and $\varphi(E)$ has zero $1$-dimensional
Hausdorff measure. If a given set $V \subset \Y \setminus \varphi(E)$
is open and compactly contained, then the restriction
$\varphi^{*}|_{ \1_{V} L^{2}( T^{*}\Y ) }$ is an isometric
isomorphism mapping onto $\1_{\varphi^{-1}(V)}L^{2}( T^{*}X )$. Since $E$ and $\varphi( E )$ have negligible measure,
this implies that the pullback operator $\varphi^{*}$ is an isometric
isomorphism.

Theorem \ref{thm:equiv_QC} implies that
there exists some function $f \in \LIP_{bs}( \X )$ for which
$f \circ \varphi^{-1} \not\in\D^{1,p}( \Y )$. It is not difficult to see that the projection $f_2 \colon \mathbb{R}^{2} \rightarrow \mathbb{R}$, $ (x,y) \mapsto y$ is such that $f_2 \circ \varphi^{-1}$ is $1$-Lipschitz. With this observation one can deduce that $f_{1} \circ \varphi^{-1} \not\in \D^{1,2}_{\loc}(\Y)$, where $f_{1}(x,y) = x$. By multiplying $f_1$ with a smooth bump function equal to one in a neighbourhood of $E$, we obtain an example of $f \in \LIP_{bs}( \X )$ for which
$f \circ \varphi^{-1} \not\in\D^{1,p}( \Y )$.
\fr}\end{remark}
\begin{remark}{\rm
In some situations, the distortion inequality $K_{O}( \varphi ) \leq K$ for a homeomorphism $\varphi \colon \X \rightarrow \Y$ does imply the corresponding inequality for the inverse. For example, if the measures on $\mm_{\X}$ and $\mm_{\Y}$ are such that balls of radius $r$ have measure comparable to $r^{p}$ and the spaces $\X$ and $\Y$ support a local $( 1, p )$-Poincaré inequality, then $K_{I}( \varphi ) \leq K'$; see \cite[Section 9]{HKST:01}. Similar conclusion holds also in the setting of two-dimensional metric surfaces admitting quasiconformal parametrizations; see \cite[Proposition 12.2]{Raj:17}.

We note that the space $\Y$ appearing in Remark \ref{rem:failure} (the one from \cite[Example 6.1]{Iko:Rom:ar}) does not admit a quasiconformal parametrization by any Riemannian surface, and such quasiconformal parametrizations of $\Y$ do not exist even locally (due to \cite[Theorem 1.2]{Iko:21:unif}). Theorem \ref{thm:equiv_QC} is closely related to some open questions regarding metric measure spaces admitting a quasiconformal parametrization by a Riemannian surface; see \cite[Question 1.6]{CR:20} and \cite[Questions 7.3 and 7.5]{Iko:21}.
\fr}\end{remark}
\end{document}